\documentclass[reqno, psamsfonts]{amsart}
\usepackage{mathabx}    
\usepackage{mathrsfs}  
\usepackage[backref,colorlinks]{hyperref} 
\usepackage{amsmath,amssymb,amsthm}
\usepackage[abbrev,msc-links]{amsrefs}   
\usepackage[utf8]{inputenc}
\usepackage{tikz}
\usepackage{amsmath}
\usepackage{enumerate}
\usepackage{lineno}
\theoremstyle{plain}
\newtheorem{thm}{Theorem}[section]

\newtheorem{prop}[thm]{Proposition}

\newtheorem{lemma}[thm]{Lemma}

\newtheorem*{clm*}{Claim}

\theoremstyle{definition}

\numberwithin{equation}{section}

\normalsize                              
\let\eps=\epsilon
\let\theta=\vartheta
\let\rho=\varrho
\let\phi=\varphi

\def\cA{{\mathcal A}}

\def\cF{{\mathcal F}}

\def\cC{{\mathcal C}}

\def\cP{{\mathcal P}}
\def\cQ{{\mathcal Q}}

\def\cS{{\mathcal S}}

\def\var{\mathop{\text{\rm Var}}\nolimits}
\def\be{\mathop{\text{\rm Be}}\nolimits}

\usepackage{enumitem} 

\let\polishlcross=\l
\def\l{\ifmmode\ell\else\polishlcross\fi}

\def\F{\mathcal{F}}

\def\aas{\emph{a.a.s.}}

\theoremstyle{plain}
\newtheoremstyle{note}% name
  {4pt}%      Space above
  {4pt}%      Space below 
  {\sl}%      Body font
  {}%         Indent amount (empty = no indent, \parindent = para indent)
  {\itshape}% Thm head font
  {.}%        Punctuation after thm head
  {.5em}%     Space after thm head: " " = normal interword space;
  {}%         Thm head spec (can be left empty, meaning `normal')

\theoremstyle{note}

\begin{document}
\title[Hamiltonicity in randomly perturbed hypergraphs]{Hamiltonicity in randomly perturbed hypergraphs}
\thanks{
JH was partially supported by FAPESP (Proc. 2014/18641-5) and Simons Foundation \#630884.
YZ is partially supported by NSF grant DMS 1700622.}

%\author[W. Bedenknecht]{Wiebke Bedenknecht}
\author{Jie Han}
%\author[Y. Kohayakawa]{Yoshiharu Kohayakawa}
\author{Yi Zhao}
%\address{Fachbereich Mathematik, Universit\"at Hamburg,
%  Bundesstra\ss{}e~55, D-20146 Hamburg, Germany}
%\email{Wiebke.Bedenknecht@uni-hamburg.de}

\address{Department of Mathematics, University of Rhode Island, 5 Lippitt Road, Kingston, RI, 02881}
\email{jie\_han@uri.edu}

\address
{Department of Mathematics and Statistics, Georgia State University, Atlanta, GA, 30303}
\email{yzhao6@gsu.edu}

\subjclass[2010]{%05C35 (primary), 05C65, 05C80 (secondary)
}
\begin{abstract}
For integers $k\ge 3$ and $1\le \ell\le k-1$, we prove that for any $\alpha>0$, there exist $\epsilon>0$ and $C>0$ such that for sufficiently large $n\in (k-\ell)\mathbb{N}$, the union of a $k$-uniform hypergraph with minimum vertex degree $\alpha n^{k-1}$ and a binomial random $k$-uniform hypergraph $\mathbb{G}^{(k)}(n,p)$ with $p\ge n^{-(k-\ell)-\epsilon}$ for $\ell\ge 2$ and $p\ge C n^{-(k-1)}$ for $\ell=1$ on the same vertex set contains a Hamiltonian $\ell$-cycle with high probability.
Our result is best possible up to the values of $\epsilon$ and $C$ and answers a question of Krivelevich, Kwan and Sudakov.
\end{abstract} 

\maketitle

%\tableofcontents

\section{Introduction}  \label{sec:intro}

\subsection{Hamiltonian cycles and random graphs}
The study of Hamiltonicity (the existence of a spanning cycle) has been a central and fruitful area in graph theory.
In particular, a celebrated result of Karp~\cite{Karp} states that the decision problem for Hamiltonicity in graphs is NP-complete.
So it is desirable to study sufficient conditions that guarantees Hamiltonicity.
Among a large variety of such results, probably the most well-known is a theorem of Dirac from 1952~\cite{Di52}: every $n$-vertex graph ($n\ge 3$) with minimum degree at least $n/2$ is Hamiltonian.

Another well-studied object in graph theory is the random graph $\mathbb{G}(n,p)$, which contains $n$ vertices and each pair of vertices forms an edge with probability $p$ independently from other pairs.
P\'osa~\cite{Posa} and Korshunov~\cite{Korshunov} independently determined the threshold for Hamiltonicity in $\mathbb{G}(n,p)$, which is around $\log n/n$.
This implies that almost all dense graphs are Hamiltonian. 
Furthermore, Bohman, Frieze and Martin~\cite{BFM} showed that for every $\alpha>0$ there is $c=c(\alpha)$ such that 
every $n$-vertex graph $G$ with minimum degree $\alpha n$ becomes Hamiltonian~\aas~after adding $c n$ random edges (we say that an event happens \emph{asymptotically almost surely}, or \aas, if it happens with probability $1-o(1)$).
This result is tight up to the value of $c$ by considering a complete bipartite graph $K_{\alpha n, (1-\alpha)n}$.
A comparison can be drawn to the notion of smoothed analysis of algorithms introduced by Spielman and Teng~\cite{SpTe}, which involves studying the performance of algorithms on randomly perturbed inputs.

\subsection{Uniform hypergraphs}
It is natural to study the Hamiltonicity of uniform hypergraphs.
Given $k\ge 2$, a $k$-uniform hypergraph (in short, a \emph{$k$-graph}) $H=(V,E)$ consists of a vertex set $V$ and an edge set $E\subseteq \binom{V}{k}$, where every edge is a $k$-element subset of $V$. 
Given a $k$-graph $H$ with a set $S$ of $d$ vertices (where $1 \le d \le k-1$) we define $N_{H} (S)$ to be the collection of $(k-d)$-sets $T$ such that $S\cup T\in E(H)$, and let $\deg_H(S):=|N_H(S)|$. 
The \emph{minimum $d$-degree $\delta _{d} (H)$} of $H$ is the minimum of $\deg_{H} (S)$ over all $d$-vertex sets $S$ in $H$.

In the last two decades, there has been a growing interest of extending Dirac's theorem to  hypergraphs.
Despite other notion of cycles in hypergraphs (e.g., Berge cycles), the following definition of cycles has become more popular recently (see surveys~\cite{RR,zsurvey}). 
For integers $1\le \ell \le k-1$ and $m\ge 3$, a $k$-graph $F$ with $m(k-\ell)$ vertices and $m$ edges is called an \emph{$\ell$-cycle} if its vertices can be ordered cyclically such that each of its edges consists of $k$ consecutive vertices and every two consecutive edges (in the natural order of the edges) share exactly $\ell$ vertices. 
A $k$-graph is called \emph{$\ell$-Hamiltonian} if it contains an $\ell$-cycle as a spanning subgraph.
Extending Dirac's theorem, the minimum $d$-degree conditions that force $\ell$-Hamiltonicity (for $1\le d, \ell\le k-1$) have been intensively studied~\cite{BMSSS1, BMSSS2, BHS, CzMo, GPW, HS, HZ2, HZ1, KKMO, KMO, KO, RoRu14, RoRuSz06, RRS08, RRS11}.
For example, the minimum $1$-degree threshold for $2$-Hamiltonicity in 3-graphs was determined asymptotically ~\cite{RRRSS}.

Let $\mathbb{G}^{(k)}(n,p)$ denote the binomial random $k$-graph on $n$ vertices, where each $k$-set forms an edge independently with probability $p$.
The thresholds for $\ell$-Hamiltonicity have been studied by Dudek and Frieze~\cite{DuFr1, DuFr2}, who proved that the asymptotic threshold is $1/n^{k-\ell}$ for $\ell\ge 2$ and $\log n/n^{k-1}$ for $\ell=1$ (they also gave a sharp threshold for $k\ge 4$ and $\ell=k-1$).

It is also natural to consider $\ell$-Hamiltonicity in randomly perturbed $k$-graphs.
In fact, Krivelevich, Kwan and Sudakov~\cite{KKS} extended the result of Bohman--Frieze--Martin~\cite{BFM} to hypergraphs.

\begin{thm}\label{thm:KKS}\cite{KKS}
Let $k\in \mathbb{N}$, and let $H$ be a $k$-graph on $n\in (k-1)\mathbb{N}$ vertices with $\delta_{k-1}(H)\geq \alpha n$. There exists a function $c_k =c_k (\alpha)$ such that for $p=c_k/ n^{k-1}$, $H\cup \mathbb G^{(k)}(n,p)$ a.a.s.~is $1$-Hamiltonian.
\end{thm}

Theorem~\ref{thm:KKS} is tight up to the value of $c_k$ (see the paragraph after Theorem~\ref{main}).
Similar results for the powers of Hamiltonian $(k-1)$-cycles were obtained by Bennett, Dudek and Frieze~\cite{BeDuFr16}, and recently by Bedenknecht, Han, Kohayakawa and Mota~\cite{BHKM}.
In addition, B\"ottcher, Montgomery, Parczyk and Person~\cite{BMPP} proved embedding results for bounded degree subgraphs in randomly perturbed graphs.
Other results in randomly perturbed graphs can be found in~\cite{BTW, KKS2, BHKMPP}.

Krivelevich, Kwan and Sudakov~\cite{KKS} asked whether Theorem~\ref{thm:KKS} can be extended to $\ell$-Hamiltonicity under minimum $d$-degree conditions for $1\le d, \ell \le k-1$. McDowell and Mycroft~\cite{McMy} found such results for $d\ge \max\{\ell, k-\ell\}$ and reiterated the question for arbitrary $d$ and $\ell$.
In this paper we solve this problem completely.
Since the minimum $1$-degree condition is the weakest among $d$-degree conditions for all $d\ge 1$, we only state and prove our result with respect to the minimum $1$-degree.

\begin{thm}\label{main}
For integers $k\ge 3$, $1\le \ell \le k-1$ and $\alpha>0$ there exist $\epsilon>0$ and an integer $C>0$ such that the following holds for sufficiently large $n\in (k-\ell)\mathbb{N}$.
Suppose $H$ is a $k$-graph on $n$ vertices with $\delta_{1}(H)\ge  \alpha {n}^{k-1}$
and
\begin{equation}\label{eq:p}
p=p(n)\ge \begin{cases}
n^{-(k-\ell)-\epsilon} &\text{ if } \ell\ge 2, \\
C n^{-(k-1)}  &\text{ if } \ell=1.
\end{cases}
\end{equation}
Then $H\cup \mathbb{G}^{(k)}(n,p)$ a.a.s.~is $\ell$-Hamiltonian.
\end{thm}

Theorem~\ref{main} is sharp up to the constants $\epsilon$ and $C$.
Indeed, given $k$ and $\ell$, let $\alpha>0$ be sufficiently small and $n\in (k-\ell)\mathbb{N}$ be sufficiently large. 
%such that $1/\alpha, \alpha n\in \mathbb{N}$.
Consider a partition $A\cup B$ of a vertex set $V$ such that $|A|=\alpha n$ and $|B|=(1-\alpha)n$.
Let $H_0$ be the $k$-graph with all $k$-tuples that intersect both $A$ and $B$ as edges.
It is easy to see that $\delta_1(H_0) = \alpha n \binom{n - \alpha n-1}{k-2}$.
Suppose $H_0\cup \mathbb{G}^{(k)}(n,p)$ \aas~contains a Hamiltonian $\ell$-cycle $C$.
Since $|A|=\alpha n$, $C$ contains at least $1/\alpha-1$ consecutive vertices in $B$.
Let $a = \lfloor (1/\alpha - 1 - \ell) / (k - \ell) \rfloor$.
Since $B$ is an independent set in $H_0$, this implies that $\mathbb{G}^{(k)}(n,p)$ \aas~contains an \emph{$\ell$-path} on $a$ edges (a $k$-graph with vertices $v_1, v_2, \dots, v_{a(k-\ell)+\ell}$ and edges $\{v_{i(k-\ell)+1}, \dots, v_{i(k-\ell)+k}\}$ for $i=0,\dots, a-1$). When $p < (1/2)^{1/a} n^{-(k-\ell)-\ell/a}$, we have $n^{\ell+a(k-\ell)} p^a <1/2$.
By Markov's inequality, with probability at least $1/2$, $\mathbb{G}^{(k)}(n,p)$ contains no $\ell$-path on $a$ edges.
When $\ell=1$, if $H_0\cup \mathbb{G}^{(k)}(n,p)$ is \aas~$\ell$-Hamiltonian, then $\mathbb{G}^{(k)}(n,p)$ \aas~contains $n/(k-1) - 2|A| > n/k$ edges (because a $1$-Hamiltonian cycle contains $n/(k-1)$ edges and each vertex is contained in at most $2$ of them). When $p < n^{-(k-1)}/(2k)$, we have $n^{k} p <n/(2k)$. 
By Markov's inequality, with probability at least $1/2$, $\mathbb{G}^{(k)}(n,p)$ contains fewer than $n/k$ edges.

%If we could reduce the $\omega(n)$ in~\eqref{eq:p} to a constant, then we would obtain another proof of Theorem~\ref{thm:KKS}. We believe that such result is true but it may require a different approach.
%We remark that although our proof uses the regularity method, the constants $\epsilon$ and $C$ in Theorem~\ref{main} are not from the regularity lemma, in particular, $1/\eps$ and $C$ only grow polynomially with $k$, $\ell$ and $1/\alpha$.

\subsection{Proof ideas}

The proof of Theorem~\ref{main} follows the \textit{absorbing method} introduced by R\"odl, Ruci\'nski, and Szemer\'edi in~\cite{RoRuSz06}. Let us define \emph{absorbers} for our problem. Given an $\ell$-path $P$, we call the first and last $\ell$ vertices two  \emph{$\ell$-ends} of $P$.
Let $H$ be a $k$-graph and $S$ be a set of $k-\ell$ vertices in $V(H)$.
We call an $\ell$-path $P$ an \emph{$S$-absorber} if $V(P)\cap S=\emptyset$ and $V(P)\cup S$ spans an $\ell$-path with the same \emph{$\ell$-ends} as $P$.

Below is a typical procedure for finding a Hamilton $\ell$-cycle in $H$ by the absorbing method.
\begin{enumerate}
\item We show that every $(k-\ell)$-subset of $V(H)$ has many absorbers (of the same fixed length). This enables us to obtain a path $P_{abs}$ of linear length such that every $(k-\ell)$-set has many absorbers on $P_{abs}$.
\item We cover most vertices of $V\setminus V(P_{abs})$ by short paths and then connect them together with $P_{abs}$ into a cycle $C$ .
\item The vertices not covered by $C$ are arbitrarily partitioned into $(k-\ell)$-sets and absorbed by $P_{abs}$ greedily.
\end{enumerate}
The proof thus has three main components: 
\begin{itemize}
\item an \emph{absorbing lemma}, which provides a family $\cA$ of vertex-disjoint short paths such that every $(k-\ell)$-set has many absorbers in $\cA$; 
\item a \emph{path cover lemma}, which allows us to cover most vertices of $V(H)$ by vertex-disjoint paths; and
\item a \emph{connecting lemma}, which allows us to connect $\cA$ into a single path $P_{abs}$ and connect the paths from the path cover lemma together.
\end{itemize}
Let $\mathbb{G}^{(k)}(n,p)\cup H$ be the underlying $k$-graph on the same vertex set $V$. Using Janson's inequality, one can derive the path cover lemma by using the edges of $\mathbb{G}^{(k)}(n,p)$. 
If we have $\delta_{k-\ell}(H)\ge \alpha \binom{n}{\ell}$, then every $(k-\ell)$-set of $V$ has many neighbors and it is not difficult to prove the absorbing lemma. If we have $\delta_{\ell}(H)\ge \alpha \binom{n}{k- \ell}$, then every $\ell$-set of $V$ has many neighbors and it is easy to prove the connecting lemma. However, our Theorem~\ref{main} only assumes $\delta_{1}(H)\ge  \alpha {n}^{k-1}$. In order to prove Theorem~\ref{main},  we ``shave'' $H$ by removing all the edges of $H$ that contain an $\ell$-set of low degree.
This results in a $k$-graph $H'$ in which every $\ell$-subset of $V$ either has a high degree or a zero degree. Our connecting lemma only connects two $\ell$-sets with high degree. 
To overcome the difficulty in absorbing, an earlier version of this paper used the hypergraph regularity method. Following the suggestion of a referee, we now give a simpler absorbing lemma without the regularity method. Note that the shaving process creates a small number of vertices that cannot be absorbed and we will cover these vertices by the path cover lemma.

The rest of the paper is organized as follows.
%In Section~\ref {sec:random} we prove some results concerning random hypergraphs. %Section~\ref{sec:abs-con} contains the Connecting and Absorbing Lemmas. We prove our main result, Theorem~\ref{main}, in Section~\ref{sec:main}.
We state and prove our lemmas in Sections 2 and 3 and prove Theorem~\ref{main} in Section 4.
%We first introduce the background of hypergraph regularity in Section 7, and give the proof of Theorem~\ref{main} for $\ell\ge 2$ in Section 8.

\medskip
\noindent\textbf{Notation.}
Given positive integers $n\ge b$, let $[n]:= \{1, 2, \dots, n\}$ and $(n)_b:=n(n-1)\cdots (n-b+1) = n!/(n-b)!$.
Given a $k$-graph $H$, we use $v_H$ and $e_H$ to denote the order and size of $H$, respectively. For two (hyper)graphs $G$ and $H$, let $G\cap H$ (or $G\cup H$) denote the (hyper)graph with vertex set $V(G)\cap V(H)$ (or $V(G)\cup V(H)$) and edge set $E(G)\cap E(H)$ (or $E(G)\cup E(H)$). Given a set $X$, $\binom Xk$ denotes the family of all $k$-subsets of $X$. A $k$-graph $(V, E)$ is \emph{complete} if $E= \binom Vk$.
Given $1\le \ell \le k$, the \emph{$\ell$-shadow of a $k$-graph $H$}, denoted by $\partial_{\ell}H$, is the collection of all $\ell$-subsets $S\subset V(H)$ that are contained in some edges of $H$.

In this paper, unless stated otherwise, we assume that the vertex sets of paths and related hypergraphs are \emph{ordered}. 
%For example, the vertices in an $\ell$-path $P_a$ are ordered naturally. 
When $A$ and $B$ are ordered sets, let $AB$  denote their concatenation. 
Given positive integers $k, \ell, a$ such that $\ell<k$, let $P_a$ denote a \emph{$k$-uniform $\ell$-path of length $a$}, that is, a $k$-graph on vertices $v_1, v_2, \dots, v_{a(k-\ell)+\ell}$ with edges $\{v_{i(k-\ell)+1}, \dots, v_{i(k-\ell)+k}\}$ for $i=0,\dots, a-1$.  %When $k$ and $\ell$ are clear from the context, we simply write $P_a$.
In general, given a $k$-graph $F$ on $\{x_1,\dots, x_s\}$ and a $k$-graph $H$, we say that \emph{an ordered subset $(v_1,\dots, v_s)$ of $V(H)$ spans a {(labeled) copy} of $F$} if $\{v_{i_1},\dots, v_{i_k}\}\in E(H)$ whenever $\{x_{i_1},\dots, x_{i_k}\}\in E(F)$.
Given integers $a\ge 1$ and $x\ge 0$, let $P_{a,x}$ denote a $k$-graph on $a(k-\ell)+\ell+2x$ vertices with an order such that the first and last $x$ vertices are isolated and the middle $a(k-\ell)+\ell$ vertices span a copy of $P_a$.
%Let $(v_1, \dots, v_{2x+\ell+a(k-\ell)})$ be an ordered set of vertices in $H$, in this paper we always consider the labeled copy of $P_{a,x}$ such that the first and last $x$ vertices are isolated and $(v_{x+1}, \dots, v_{x+\ell+a(k-\ell)})$ span a copy of $P_a$ following the linear order.
%Let $T_{a,x}$ be the ordered $k$-graph on $2x+\ell+a(k-\ell)$ vertices such that the first $x$ and the last $x$ vertices are isolated and the $\ell+a(k-\ell)$ vertices in the middle span a copy of $P_a$.

Throughout the rest of the paper, we write $\alpha \ll \beta \ll \gamma$ to mean that we can choose the positive constants
$\alpha, \beta, \gamma$ from right to left. More
precisely, there are increasing functions $f$ and $g$ such that, given
$\gamma$, whenever $\beta \leq f(\gamma)$ and $\alpha \leq g(\beta)$, the subsequent statement holds.
Hierarchies of other lengths are defined similarly.
%Moreover, when we use variables in their reciprocal form in the hierarchy, we implicitly assume that these variables are integers.

Throughout the paper we omit floor and ceiling functions when they are not crucial.

\section{Subgraphs in random hypergraphs}\label{sec:random}

In this section we introduce some results related to binomial random $k$-graphs (similar ones can be found in~\cite{BHKM}).
Our main tools are Janson's inequality (see, e.g.,~\cite[Theorem 2.14]{JLR}) and Chebyshev's inequality.
%In the following lemma,~\eqref{eq:1} follows from Chebyshev's inequality and~\eqref{eq:2} follows from Janson's inequality.
%Let $I_i\in \be(p_i)$, $i\in \mathcal{I}$, be a finite family of Bernoulli random variables having a dependency graph $L$, i.e., if $ij\in E(L)$ then $I_i$ and $I_j$ are dependent. Let $X=\sum_i I_i$ and $\lambda =\mathbb{E}X=\sum_i p_i$. Moreover, write $i\sim j$ if $ij\in E(L)$, and let $\Delta = \sum_{i\sim j}\mathbb{E}[I_iI_j]$ and $\delta =max_i \sum_{k\sim i}p_k$ and $\overline{\Delta}=\lambda+2\Delta$. 
%Then we have the following estimates on $X$.

%\begin{lemma}
%Let $X$ be a random variable. Then
%$$
%\mathbb{P}(X\ge 2\mathbb{E}[X)] \le \frac{\mathbb{E}[X)+\Delta}{\mathbb{E}[X)^2} .
%$$
%\end{lemma}

%Given a family of random variables $\{I_i\}_{i\in \mathcal{I}}$, a \emph{dependency graph $L$ for $\{I_i\}_{i\in \mathcal{I}}$} is a graph with vertex set $\mathcal I$ and $ij\in E(L)$ if and only if $I_i$ and $I_j$ are dependent. 
%We write $i\sim j$ if $ij\in E(L)$ and define $\Delta = \Delta(\{I_i\}_{i\in \mathcal{I}}) = \sum_{i\sim j}\mathbb{E}[I_iI_j]$.
%Let $X:=\sum_{S\in \cS} I_S$ be the sum of a family of Bernoulli random variables and let $\lambda = \mathbb E(X)$.
%Let $\Delta_{X} := \sum_{S\cap T\neq\emptyset}\mathbb{E}(I_S I_T)$, where the sum is over ordered pairs $S, T\in \cS$.

We first recall Janson's inequality.
Let $\Gamma$ be a finite set and let $\Gamma_p$ be a random subset of $\Gamma$ such that each element of $\Gamma$ is included independently with probability $p$.
Let $\cS$ be a family of non-empty subsets of $\Gamma$ and for each $S\in \cS$, let $I_S$ be the indicator random variable for the event $S\subseteq \Gamma_p$.
%$I_A = \textbf{1}[A\subseteq \Gamma_p]$.
Thus each $I_S$ is a Bernoulli random variable $\be(p^{|S|})$.
%Given a family of random variables $\{I_i\}_{i\in \mathcal{I}}$, for $i, j\in \mathcal I$, we write $i\sim j$ if and only if $I_i$ and $I_j$ are dependent. 
%Let $\Delta = \sum_{i\sim j}\mathbb{E}[I_iI_j]$, where the sum is over unordered pairs.
Let $X:=\sum_{S\in \cS} I_S$ and $\lambda = \mathbb E(X)$.
Let $\Delta_{X} := \sum_{S\cap T\neq\emptyset}\mathbb{E}(I_S I_T)$, where the sum is over not necessarily distinct ordered pairs $S, T\in \cS$.
Then Janson's inequality says that for any $0\le t\le \lambda$,
\begin{equation}
\mathbb{P}(X\leq \lambda -t)\leq \exp \left ( -\dfrac{t^2}{2\Delta_{X}}\right ). \label{eq:2}
\end{equation}
Next note that $\var(X)=\mathbb E(X^2) - \mathbb E(X)^2 \le \Delta_X$.
Then by Chebyshev's inequality,
\begin{equation}
\mathbb{P}(X\ge 2\lambda) \le \frac{\var(X)}{\lambda^2} \le \frac{\Delta_X}{\lambda^2}. \label{eq:1}
\end{equation}

Consider the random $k$-graph $\mathbb{G}^{(k)}(n, p)$ on an $n$-vertex set $V$.
Note that we can view $\mathbb{G}^{(k)}(n, p)$ as $\Gamma_p$ with $\Gamma = \binom{V}k$.
Let $\Phi_F = \Phi_F(n, p)= \min \{n^{v_H} p^{e_H}: H\subseteq F, e_H>0\}$.
%The following simple proposition can be found in~\cite[Proposition 2.2]{BHKM} -- we include its proof for completeness.
The following simple proposition is useful.

\begin{prop}\label{prop:1}%\cite{BHKM}
Let $F$ be a $k$-graph with $s$ vertices and $f$ edges and let $G:=\mathbb{G}^{(k)}(n, p)$ on $V$.
Given a family $\cA$ of ordered $s$-subsets of $V$,
%Given a family $\cA$ of ordered $s$-subsets of $V$, let $\cS_\cA$ consist of the edge sets of the labeled copies of $F$ spanned on $A$ in the complete $k$-graph on $V$ for all $A\in \cA$.
let $X_\cA=\sum_{A\in \cA} I_A$, 
%such that for each ordered $s$-set $A\in \cA$, 
where $I_A$ is the Bernoulli random variable for the event that $A$ {spans a labeled copy of $F$ in} $G$.
%Let $\lambda =\mathbb{E}(X)$. 
%Let $\cA$ and $\cS_\cA$ as defined in Lemma~\ref{lm:1}.
Then $\Delta_{X_\cA} \leq 2^s s! \, n^{2s}p^{2f}/\Phi_F$.
\end{prop}

\begin{proof}
Fix $1\le i\le s$. There are $\binom{s}{i} (s)_i$ ways that two labeled $s$-sets share exactly $i$ vertices. Fixing two such $s$-sets, there are $(n)_{2s - i}$ ways mapping their $2s-i$ vertices into  $V$.
Let $f_i$ denote the maximum number of edges of an $i$-vertex subgraph of $F$.  We have
\[
\Delta_{X_\cA} \le \sum_{i=1}^s \binom{s}{i} (s)_i (n)_{2s - i} p^{2f - f_i} \le \sum_{i=1}^s \binom si s! \, n^{2s-i} p^{2f - f_i} 
\le 2^{s} s! \, n^{2s}p^{2f}/\Phi_F. \qedhere
\]
\end{proof}

The next two lemmas gather all the properties of $\mathbb{G}^{(k)}(n,p)$ that we will use.

\begin{lemma}\label{lm:gnp}
Let $k, \ell, a, x\in \mathbb Z$ such that $k\ge 3, 1\le \ell\le k-1$, $a\geq 1$, and $0\le x\le k$. Write $b = b(x) =2x+\ell+(k-\ell)a$.
Suppose $0<\epsilon\le \ell/(3a)$ and $1/n \ll 1/C \ll \gamma, 1/a, 1/k$.
%, and let $\F_1, \dots, \F_t$ be $t\le n^{2k}$ families of $\gamma n^{b}$ ordered $b$-sets on $V$.
Let $G=\mathbb{G}^{(k)}(n, p)$ be a random $k$-graph with vertex set $V$, where $p$ satisfies~\eqref{eq:p}. Then the following properties hold.
%\begin{equation}\label{eq:p}
%p=p(n)\ge \begin{cases}
%n^{-(k-\ell)-\epsilon} &\text{ if } \ell\ge 2, \\
%C/n^{k-1}  &\text{ if } \ell=1.
%\end{cases}
%\end{equation}
%Then the following properties hold.
\begin{enumerate}
\item Let $L$ be a family of $\ell$-sets in $V(G)$ and in addition assume $a\ge \ell/(k-\ell)$. 
Then for every $R, V^*\subseteq V(G)$ such that $|V^*|\ge \gamma n$ and $|L\cap \binom{R}{\ell}|\ge \gamma n^{\ell}$, with probability at least $1-\exp(-3 n)$, $G$ contains a copy of $P_{a}$ whose $\ell$-ends are in $L\cap \binom{R}{\ell}$ and whose other vertices are from $V^*$. 
Moreover, this property holds for all choices of $R$ and $V^*$ simultaneously with probability $1-o(1)$. \label{item-I}
%\item With probability at least $1-\exp(-n)$, for every $i\in [t]$, at least $(\gamma/2) p^{a}  n^{b}$ members of $\F_i$ span copies of $P_{a,x}$. \label{item-II}
\item With probability at least $1-o(1)$, at most $2 p^{a} n^{b}$ ordered $b$-subsets of $V(G)$ span copies of $P_{a,x}$. \label{item-III}
\item With probability at least $1-o(1)$, $G$ contains at most $4 b^2 n^{2b-1} p^{2a}$ pairs of overlapping (i.e., not vertex-disjoint) copies of $P_{a,x}$. \label{item-IV}
\end{enumerate}
%Furthermore, if $G=\mathbb{G}^{(k)}(n,p)$ is such that $p\ge n^{-(k-\ell)-\eps}$, then
%\begin{enumerate}[resume*]
%\item with probability at least $1-\exp(-n^{1/3})$, for every $i\in [t]$, at least $(\gamma/2) p^{a}  n^{b}$ members of $\F_i$ span copies of $P_{a,x}$. \label{item-II}
%\end{enumerate}
\end{lemma}

\begin{proof}%[Proof of Lemma~\ref{lm:gnp}]

%We first compute $\Phi_{P_{a,x}}$.
Note that if $H$ is a subgraph of $P_{a,x}$, then $v_H\ge \ell+(k-\ell)e_H$. Thus,
\begin{equation}
\label{eq:phi1}
\Phi_{P_{a,x}} = \min_{1\le e_H\le a} n^{v_H} p^{e_H} \ge \min_{1\le e_H\le a} n^{\ell+(k-\ell)e_H} p^{e_H} = n^{\ell} \min_{1\le e_H\le a} (n^{k-\ell} p)^{e_H} \ge 
\begin{cases}
n^{\ell - a\epsilon} & \text{if } \ell\ge 2,\\
C n & \text{if }\ell=1,
\end{cases}
\end{equation}
where we used~\eqref{eq:p} in the last inequality.
Since $\eps\le \ell/(3a)$, $\Phi_{P_{a,x}}\ge C n$ holds for all $\ell$.

Given a family $\cA$ of ordered $b$-sets of vertices in $V$, let $\cS$ consist of the edge sets of the labeled copies of $P_{a,x}$ spanned on $A$ in the complete $k$-graph on $V$ for all $A\in \cA$. Let $X_\cA=\sum_{S\in \cS} I_S$, where $I_S$ is the indicator variable for the event $S\subseteq E(G)$ (thus $X_\cA$ counts the number of $A\in \cA$ that spans a copy of $P_{a,x}$ in $G$).
%we can apply~\eqref{eq:2} to $X_\cA$.
%let $X_\cA=\sum_{A\in \cA} I_A$, 
%such that for each $A\in \cA$, 
%where $I_A$ is the Bernoulli random variable for the event that $A$ {spans a labeled copy of $P_{a,x}$ in} $G$.
Since $\Phi_{P_{a,x}}\ge C n$, Proposition~\ref{prop:1} implies that
\begin{equation}\label{eq:phi}
\Delta_{X_\cA} \leq 2^{b} b! \, n^{2b}p^{2a}/\Phi_{P_{a,x}}\le ( 2^b b! /C) n^{2b-1}p^{2a} \le (\gamma^{2b}/24) n^{2b-1}p^{2a}
\end{equation}
because $1/C\ll \gamma, 1/a, 1/k$.

For~\eqref{item-I}, fix such a choice for $R$ and $V^*$ and let $x=0$ and $b= \ell+(k-\ell)a$. 
Let $\mathcal A$ be the family of all ordered $(\ell+(k-\ell)a)$-sets in $V(G)$ whose first and last $\ell$ vertices are in $L\cap \binom{R}{\ell}$ and all other vertices are from $V^*$. 
Then $|\cA| \ge (\gamma n^\ell)^2 (\gamma n)^{(k-\ell)a-\ell}/2 \ge (\gamma n)^{b}/2$.
Recall that $X_\cA$ counts the number of $A\in \cA$ that spans a copy of $P_a$ in $G$.
%Let $X_{\mathcal A}$ be the random variable that counts the number of members in $\mathcal A$ that span a copy of $P_{a,0}$ and 
Then $(\gamma n)^{b} p^a/2\le \mathbb{E}(X_{\mathcal A})\le n^{b} p^a$.
%By~\eqref{eq:p}, we have $p\ge C n^{-(k-\ell)}$ and $n^{b-1} p^a = (n^{k-\ell} p)^a n^{\ell - 1} \ge C^a$.
%Consequently,
%\begin{equation}
%\label{eq:nb}
%n^b p^a \le (\gamma^{2b}/24) n^{2b-1}p^{2a}
%\end{equation}
%because $1/C\ll \gamma, 1/a, 1/k$.
%Note that $\delta:=\max_{i\in[r]}\sum_{j\sim i} p^{a} \le (\ell+(k-\ell)a)n^{\ell+(k-\ell)a-k}p^{a}$, because $i\sim j$ only if $|A_i\cap A_j|\ge k$, for $A_i, A_j\in \mathcal A$.
By~\eqref{eq:2} and \eqref{eq:phi}, we have 
\[
\mathbb{P} (X_\cA=0) \le 
\exp \left( -\frac{\mathbb{E}(X_{\mathcal A})^2}{2\Delta_{X_{\cA}}} \right)
\le \exp \left( -\frac{(\gamma n)^{2b} p^{2a}/4}{(\gamma^{2b}/12) n^{2b-1}p^{2a} } \right)
= \exp( - 3 n).
\]
%because $1/C\ll \gamma, 1/a, 1/k$.
The second part of~\eqref{item-I} follows from the union bound because there are at most $2^n$ choices for each of $R$ and $V^*$ and $2^n \cdot 2^n \cdot \exp(- 3 n) \le \exp(-n)$.

For \eqref{item-III}, %let $\cF$ be the family of all ordered $b$-sets in $V$ and 
let $X_2$ be the random variable that counts the number of labeled copies of $P_{a,x}$ in $G$. 
Then $\mathbb E(X_2)= {(n)}_{b}p^{a}$. By~\eqref{eq:1} and \eqref{eq:phi}, we have
\[
\mathbb{P}(X_2\ge 2p^an^b)\leq \mathbb{P}(X_2\ge 2\mathbb{E}(X_2))\le \frac{\Delta_{X_2 }}{\mathbb{E}(X_2)^2} \le \frac{ (\gamma^{2b}/24) n^{2b-1}p^{2a} }{((n)_{b} \, p^{a})^2} = o(1).
\]
%This shows \eqref{item-III}.

For \eqref{item-IV}, 
%let $\cQ$ be the family of all two overlapping copies of $P_{a,x}$ in the complete $k$-graph on $V$.
let $\cQ$ consist of edge sets of all overlapping pairs of $P_{a,x}$ in the complete $k$-graph on $V$. 
Let $Y=\sum_{Q\in \cQ}I_Q$, where $I_Q$ is the indicator variable for the event $Q\subseteq E(G)$.
%let $Y$ be the random variable that counts the number of pairs of overlapping copies of $P_{a,x}$ in $G$.
We first estimate $\mathbb{E}(Y)$.
For $X_2$ defined above, we have $\Delta_{X_2} = \mathbb{E}(\sum_{Q}I_Q)$, where the sum is over all $Q\in \cQ$ whose two copies of $P_{a,x}$ share at least one edge.
%Now assume two copies of $P_{a,x}$ overlap but do not share any edge. 
%Then they induce at most $2b-1$ vertices and exactly $2a$ edges. 
As shown in the proof of Proposition~\ref{prop:1}, for $1\le i\le b$, there are $(n)_{2b - i} \binom b i (b)_i$
%\[
%\binom{n}{2b-i} (2b-i)_{b}\binom{b}{i} b! = (n)_{2b - i} \binom b i (b)_i \le (n)_{2b - i}(b)_i ^2
%\]
members of $\cQ$ whose two copies of $P_{a,x}$ share exactly $i$ vertices. 
Hence $\mathbb{E}(Y) \ge (n)_{2b - 1} b^2 \, p^{2a} \ge {n}^{2b - 1} p^{2a} b^2/2$. Since $\sum_{2\le i\le b} (n)_{2b-i} (b)_i^2 \le  n^{2b-1}/2$, using \eqref{eq:phi}, we derive that 
\[
\mathbb{E}(Y)\le (n)_{2b - 1} b^2 \, p^{2a} + (n^{2b-1}/2)\, p^{2a} + \Delta_{X_{2}} \le 2b^2 {n}^{{2b - 1}} p^{2a}.
\]
%By considering all pairs that overlap at exactly one vertex, we also have $\mathbb{E}[Y]\ge {n}^{{2b - 1}} p^{2a}/2$.

We next compute $\var(Y)$. 
%Let $\cQ$ consist of the edges sets of all two overlapping copies of $P_{a,x}$ in the complete $k$-graph on $V$ (thus $Y=\sum_{Q\in \cQ}I_Q$). For each $Q\in \cQ$, let $S_Q$ be the subgraph of $G$ induced by $Q$.
%Let $\cQ$ be the family of all two overlapping copies of $P_{a,x}$ in the complete $k$-graph on $V$
%be the family of all pair of overlapping copies of $P_{a,x}$ in the complete $k$-graph on $V$ 
%(thus $Y=\sum_{Q\in \cQ}I_Q$, where $I_Q$ is the indicator variable for the event $Q\subseteq G$). 
%For each $Q\in \cQ$, let $S_Q$ be the subgraph of $G$ induced by $Q$.
%is a set of edges which represent a set of size between $b$ and $2b-1$ spans an overlapping pair of copies of $P_{a,x}$ in $\mathbb{G}^{(k)}(n,p)$.
% label all vertex sets of size between $b$ and $2b-1$ as $\{S_i\}_{i\in \mathcal I}$, and let $I_i'$ be the indicator variable for the event that $S_i$ spans an overlapping pair of copies of $P_{a,x}$ in $\mathbb{G}^{(k)}(n,p)$.
%So we have $Y=\sum I_i'$. 
%Let $\Delta_\cQ = \sum_{Q\cap R\neq\emptyset} \mathbb{E}(I_Q I_{R})$.
%, where $i\sim j$ denotes that $I_i'$ and $I_j'$ are correlated, i.e., $S_i$ and $S_j$ share at least one edge.
For each $Q\in \cQ$, let $S_Q$ denote the $k$-graph induced by $Q$ (thus $S_Q$ is the union of two overlapping copies of $P_{a,x}$).
Fix two $Q, R\in \cQ$ such that $Q\cap R \ne \emptyset$. 
We write $S_Q=T_1\cup T_2$ and $S_R=T_3\cup T_4$, where $T_i$'s are copies of $P_{a,x}$ such that $E(T_1)\cap E(T_3)\neq\emptyset$.
Define $H_1:=T_1\cap T_2$, $H_2:=(T_1\cup T_2)\cap T_3$ and $H_3:=(T_1\cup T_2\cup T_3)\cap T_4$.
Since $V(T_1)\cap V(T_2)\ne \emptyset$, $V(T_3)\cap V(T_4)\ne \emptyset$, and $E(T_1)\cap E(T_3)\ne \emptyset$,
%Since $T_1$ overlap $T_2$, $T_3$ overlap $T_4$, and , 
it follows that $v_{H_i}\ge 1$ for $i=1,2,3$.
We claim that $n^{v_{H_i}}p^{e_{H_i}}\ge n$ for $i=1,2,3$.
Indeed, since each $H_i$ is a subgraph of $P_{a,x}$, if $e_{H_i}\ge 1$, then by~\eqref{eq:phi}, $n^{v_{H_i}}p^{e_{H_i}}\ge \Phi_{P_{a,x}}\ge C n$; otherwise $e_{H_i}=0$ and then we have  $n^{v_{H_i}}p^{e_{H_i}} = n^{v_{H_i}} \ge n^1=n$.
Consequently,
\begin{equation}\label{eq:estnp}
n^{v_{H_1}}p^{e_{H_1}}\cdot n^{v_{H_2}}p^{e_{H_2}}\cdot n^{v_{H_3}}p^{e_{H_3}} \ge n^{3}.
\end{equation}

Let $D=D(b,k)$ be the number of choices for $H_1, H_2, H_3$. Fix some $H_1, H_2, H_3$. Let
$
\Delta_{H_1, H_2, H_3}= \sum_{Q, R} \mathbb{E}(I_Q I_R),
$
where the sum is over all $Q, R\in \cQ$ with $Q\cap R\neq\emptyset$ that generate the given $H_1, H_2, H_3$. 
It is easy to see that the sum contains at most
\[
%\binom n{4b-v_{H_1} - v_{H_2} -v_{H_3}} (4b-v_{H_1} - v_{H_2} -v_{H_3})_{b}^4 < n^{4b-(v_{H_1} + v_{H_2} + v_{H_3})} (4b)^{3b}
\binom{b}{v_{H_1}} (b)_{v_{H_1}} \binom{b}{v_{H_2}} (2b-v_{H_1})_{v_{H_2}} \binom{b}{v_{H_3}} (3b- v_{H_1} - v_{H_2})_{v_{H_3}} (n)_{4b-v_{H_1} - v_{H_2} -v_{H_3}} \le 2^{3b} (3b)! n^{4b-(v_{H_1} + v_{H_2} + v_{H_3})} 
\]
terms.
%For fixed $H_1, H_2, H_3$, observe that any such $S_Q\cup S_R$ contains $4b-(v_{H_1} + v_{H_2} + v_{H_3})$ vertices and $4a - (e_{H_1} + e_{H_2} + e_{H_3})$ edges.
%Moreover, there are at most $(4b)!2^{2b}(2b)!$ ways to label these vertices to determine $\{S_Q, S_R\}$.
Together with~\eqref{eq:estnp}, we obtain that
\[
\Delta_{H_1, H_2, H_3}= 
\sum_{Q, R} \mathbb{E}(I_Q I_R) \le 2^{3b} (3b)! n^{4b-(v_{H_1} + v_{H_2} + v_{H_3})} p^{4a - (e_{H_1} + e_{H_2} + e_{H_3})} \le 2^{3b} (3b)! n^{4b - 3} p^{4a}.
\]
%So $\Delta_\cQ = \sum_{H_1, H_2, H_3} \Delta_{H_1, H_2, H_3}$.
Consequently,
\[
\Delta_Y = \sum_{H_1, H_2, H_3} \Delta_{H_1, H_2, H_3} \le D 2^{3b} (3b)! n^{4b-3}p^{4a}.
\]
By~\eqref{eq:1}, we derive that
\[
\mathbb{P}\big(Y\ge 4 b^2 n^{2b-1}p^{2a}) \le \frac{\Delta_Y}{\mathbb{E}(Y)^2} \le \frac{D 2^{3b} (3b)! n^{4b-3}p^{4a}}{({n}^{{2b-1}} p^{2a} b/2)^2} = o(1).
\]
This confirms \eqref{item-IV}.
%
%For \eqref{item-II}, by~\eqref{eq:phi1} and $\eps<\ell/(2a)$, we have $\Phi_{P_{a,x}} \ge n^{\ell-a\eps} \ge \sqrt n$.
%%consider a fixed collection of distinct $\gamma_2n^b$ ordered $b$-sets of vertices of $\mathbb G^{(k)}(n,p)$.
%Fix $i\in [t]$ and let $X_{\F_i}$ be the random variable that counts the number of the members of $\F_i$ that span copies of $P_{a,x}$. 
%%We will apply Lemma~\ref{lm:1}. 
%By~\eqref{eq:phi}, we have $\Delta_{X_{\cF_i}} \leq 2^{b} b! \, n^{2b}p^{2a}/\sqrt n$ and note that $\mathbb{E}(X_{\cF_i})=\gamma n^{b}p^a$.
%% and $\delta \le n^{b-k}p^{a}$.
%By~\eqref{eq:2}, we have 
%\[
%\mathbb{P}\big(X_{\cF_i}\leq (\gamma/2)n^{b}p^{a}\big) \le \exp \left( -\frac{ (\mathbb{E}(X_{\cF_i})/2)^2 }{2\Delta_{X_{\cF_i}}} \right)\le \exp \left( -\frac{(\gamma/2)^2n^{2b}p^{2a}}{2^{b} b! \, n^{2b}p^{2a}/\sqrt n} \right) \le \exp(- 2n^{1/3}).
%\]
%Since $n^{2k} \exp(- 2n^{1/3})\le \exp(- n^{1/3})$,~\eqref{item-II} follows from the union bound.
\end{proof}

%The following lemma studies the paths of fixed length in $G=\mathbb{G}^{(k)}(n,p)$.
In Lemma~\ref{lm:gnp} we assume that $p$ satisfies \eqref{eq:p} and obtain that $\Phi_{P_{a,x}}\ge C n$. This is necessary for Part~\eqref{item-I}, in which we use the union bound on $2^n$ events. When there are only  
polynomially many events, it suffices to have $\Phi_{P_{a,x}}\ge n^c$ for some $0<c<1$, which occurs when $p\ge n^{-(k-\ell)-\eps}$ (for all $\ell\ge 1$) and $\epsilon< \ell/a$. We use this weaker condition on $p$ in the following lemma because we only have this condition in the proof of Lemma~\ref{lm:almost}.
%which is needed for Lemmas~\ref{prop:prob} and~\ref{lm:almost}.

\begin{lemma}\label{lm:gnp2}
Let $k, \ell, a, x\in \mathbb Z$ such that $k\ge 3, 1\le \ell\le k-1$, $a\geq 1$, and $0\le x\le k$. Write $b = b(x) =2x+\ell+(k-\ell)a$.
Suppose $0< \epsilon\le \ell/(2a)$ and $1/n \ll \gamma, 1/a, 1/k$.
Let $V$ be an $n$-vertex set, and let $\F_1, \dots, \F_t$ be $t\le n^{2k}$ families of $\gamma n^{b}$ ordered $b$-sets on $V$.
Suppose $G=\mathbb{G}^{(k)}(n,p)$ with $p\ge n^{-(k-\ell)-\eps}$, then with probability at least $1-\exp(-n^{1/3})$, for all $i\in [t]$, at least $(\gamma/2) p^{a}  n^{b}$ members of $\F_i$ span copies of $P_{a,x}$. %\label{item-II}
\end{lemma}

\begin{proof}
By~\eqref{eq:phi1} and $\eps\le \ell/(2a)$, we have $\Phi_{P_{a,x}} \ge n^{\ell-a\eps} \ge \sqrt n$.
%consider a fixed collection of distinct $\gamma_2n^b$ ordered $b$-sets of vertices of $\mathbb G^{(k)}(n,p)$.
Fix $i\in [t]$ and let $X_{\F_i}$ be the random variable that counts the number of the members of $\F_i$ that span copies of $P_{a,x}$. 
%We will apply Lemma~\ref{lm:1}. 
By~\eqref{eq:phi}, we have $\Delta_{X_{\cF_i}} \leq 2^{b} b! \, n^{2b}p^{2a}/\sqrt n$ and note that $\mathbb{E}(X_{\cF_i})=\gamma n^{b}p^a$.
% and $\delta \le n^{b-k}p^{a}$.
By~\eqref{eq:2}, we have 
\[
\mathbb{P}\big(X_{\cF_i}\leq (\gamma/2)n^{b}p^{a}\big) \le \exp \left( -\frac{ (\mathbb{E}(X_{\cF_i})/2)^2 }{2\Delta_{X_{\cF_i}}} \right)\le \exp \left( -\frac{(\gamma/2)^2n^{2b}p^{2a}}{2^{b} b! \, n^{2b}p^{2a}/\sqrt n} \right) \le \exp(- 2n^{1/3}).
\]
Since $n^{2k} \exp(- 2n^{1/3})\le \exp(- n^{1/3})$, the result follows from the union bound.
\end{proof}

\section{Lemmas}

In this section we prove all the lemmas that are needed for the proof of Theorem~\ref{main}.

Since we assume $\delta_1(H)\ge \alpha n^{k-1}$, unless $\ell=1$, the $k$-graph $H$ may contain some $\ell$-sets $S$ whose degree is too low to be used for connection. To overcome this, we simply delete all edges that contain $S$.
The following lemma reflects this ``shaving'' process.

\begin{lemma}\label{lem:shave}
Let $0<\eta \le \alpha, 1/k$.
Let $H$ be an $n$-vertex $k$-graph with $\delta_1(H)\ge \alpha n^{k-1}$.
Then there exists a spanning subgraph $H'$ of $H$, satisfying the following properties.
\begin{enumerate}
\item $e(H')\ge {\alpha} n^{k}/(2k)$.
\item $\deg_{H'}(v)\ge 2\alpha n^{k-1}/3$ for all but at most $3k\eta^2 n/\alpha$ vertices of $H$.
%The number of vertices $v$ in $H'$ such that $\deg_{H'}(v)\le \alpha n^{k-1}/2$ is at most $2k\eta^2 n/\alpha$.
\item For every $\ell$-set $S$ of $V(H)$, either $\deg_{H'}(S)=0$ or $\deg_{H'}(S) \ge \eta^2 n^{k-\ell}$.
\end{enumerate}
\end{lemma}

\begin{proof}
Starting from $H$, we iteratively do the following. If the current $k$-graph contains an $\ell$-set $S$ whose degree is less than $\eta^2 n^{k-\ell}$, then we delete all the edges containing $S$.
Clearly the iteration lasts at most $\binom{n}{\ell}$ steps.
Let $H'$ be the resulting $k$-graph, then (3) holds.
Since we deleted at most $\eta^2 n^{k-\ell}$ edges in each step, we have $e(H) - e(H')\le \binom{n}{\ell}\eta^2 n^{k-\ell}\le {\alpha} n^{k}/(2k)$.
Together with $e(H)\ge (n/k) \alpha n^{k-1}$, (1) follows.
For (2), let $V_0$ be the set of vertices $v$ in $H'$ such that $\deg_{H'}(v)\le 2\alpha n^{k-1}/3$, then since $\delta_1(H)\ge \alpha n^{k-1}$, we have
\[
|V_0| \cdot \frac13\alpha n^{k-1} \le k(e(H) - e(H')) \le k\binom{n}{\ell}\eta^2 n^{k-\ell} \le k\eta^2 n^k.
\]
Thus $|V_0|\le 3k\eta^2 n/\alpha$ and (2) holds.
\end{proof}

%\subsection{A result on sampling}
\medskip
We recall the following Chernoff's inequality (see, e.g.,~\cite{JLR}).
For $x >0$ and a binomial random variable $X=Bin(n, \zeta)$, it holds that 
\begin{align}
\mathbb P(X\ge n\zeta + x)< e^{-x^2/(2 n\zeta + x/3)} \quad \text{and} \quad \mathbb{P}(X\le n\zeta - x)< e^{-x^2/(2 n\zeta)}. \label{eq:cher2}
\end{align}

The following lemma helps us to build connectors and absorbers.

\begin{lemma}\label{prop:prob}
%Write $b:=2x+\ell+(k-\ell)a$.
%Given $b\ge 3$ and $0<\beta<1$, 
Let $k, \ell, a, x, b,\epsilon$ be as in Lemma~\ref{lm:gnp}.
Suppose $1/n \ll1/C \ll \beta, 1/b$.
Let $V$ be an $n$-vertex set, and let $\F_1, \dots, \F_t$ be $t\le n^{2k}$ families of $24\beta n^{b}$ ordered $b$-sets on $V$.
Suppose $G=\mathbb{G}^{(k)}(n, p)$ on $V$ and $p$ satisfies~\eqref{eq:p}.
Then a.a.s.~there exists a family $\F\subseteq \bigcup_{i\in [t]} \F_i$ of at most $\beta n$ disjoint ordered $b$-sets such that $|\F_i\cap \F|\ge \beta^2 n/b^2$ for each $i\in [t]$, and each member of $\F$ spans a labeled copy of $P_{a,x}$ in $G$.
%Moreover, each ordered set $F=(v_1,\dots, v_{b})\in \F'$ span a labeled copy of $T_{a,x}$ in $G$ such that $(v_{x+1}, \dots, v_{x+\ell+(k-\ell)a})$ span a .
\end{lemma}

\begin{proof}
In $G=\mathbb{G}^{(k)}(n, p)$, let $\mathcal{T}$ be the set of all ordered $b$-sets on $V$ that span copies of $P_{a,x}$. By Lemma~\ref{lm:gnp}~\eqref{item-III} and~\eqref{item-IV}, Lemma~\ref{lm:gnp2} and the union bound, \aas~the following properties hold simultaneously .
\begin{itemize}
\item $|\F_i\cap \mathcal T|\ge 12\beta p^{a} n^{b}$ for all $i\in [t]$;
\item $|\mathcal T|\le 2 p^{a} n^{b} $;
\item there are at most $4 b^2 p^{2a}n^{2b-1}$ pairs of overlapping members of $\mathcal T$.
\end{itemize}

Next we select a random set $\mathcal{F'}\subset \mathcal T$ by including each member of $\mathcal T$ independently with probability $q:=\beta/(2 b^2 n^{b-1}p^{a})$.
Because of \eqref{eq:cher2} (for (i) and (ii) below) and Markov's inequality (for (iii)), there exists such a family $\mathcal F'$ satisfying the following properties:
\begin{itemize}
\item[(i)] $|\cF_i\cap \F'|\ge 12\beta (q/2) p^{a} n^{b} = 3 \beta^2 n/b^2$ for all $i\in [t]$;
\item[(ii)] $|\mathcal F'|\le 2q |\mathcal T|\le \beta n$;
%\item[(ii)] for every $i\in [t]$, there are at least $12\beta (q/2) p^{a} n^{b} = 3 \beta^2 n$ ordered copies of $P_{a,x}$ in $\mathcal F'\cap \cF_i$;
\item[(iii)] there are at most $ 8b^2 q^2 n^{2b-1} p^{2a}= 2\beta^2 n/b^2$ pairs of overlapping members of $\mathcal F'$.
\end{itemize}
By deleting one ordered $b$-set from each overlapping pair and all ordered $b$-sets not in $\bigcup_{i\in [t]} \F_i$, we obtain a collection $\mathcal F$ of disjoint ordered $b$-sets such that $|\mathcal F|\le \beta n$, and for every $i\in [t]$, 
$|\F_i\cap \F|\ge 3\beta^2 n/b^2 - 2\beta^2 n/b^2= \beta^2 n/b^2$.
Moreover, since $\mathcal{F}\subseteq \mathcal T$, each member of $\F'$ spans a labeled copy of $P_{a,x}$ in $G$.
\end{proof}

\medskip
We now prove a connecting lemma that provides connectors for any two $\ell$-sets with large degree.
%In particular, the connecting lemma contains two deterministic $k$-graphs $H_1, H_2$, because later we will need to apply Lemma~\ref{lem:shave} independently twice, and will (also) be interested in connecting two $\ell$-sets one with large degree in $H_1$ and one with large degree in $H_2$.
Throughout the rest of the paper, let
\[
t_1:=\lceil \ell /(k-\ell) \rceil, \quad t_2:=t_1(k-\ell)-\ell, \quad \text{and} \quad t_3: = 3t_1(k-\ell)-\ell.
\]
%Clearly $t_1(k-\ell)+t_2=k$.
Given a $k$-graph $H$, we say that an ordered $t_3$-set $C$ \textit{connects} two ordered $\ell$-sets $A$ and $B$ if $C\cap A=C\cap B=\emptyset$ and the concatenation $ACB$ spans an $\ell$-path.
Note that in this case, $C$ spans a copy of $P_{t_1, t_2}$ in $H$.
%We next present our connecting lemma. 
%use ordered copies of $P_{t_1, t_2}$ to connect ordered $\ell$-sets.
%Recall that $P_{t_1, t_2}$ denote the ordered $k$-graph on $3t_1(k-\ell)-\ell$ vertices, in which the first and last $t_2$ vertices are isolated, and the $t_1(k-\ell)+\ell$ vertices in the middle span an ordered copy of $P_{t_1}$.
%Note that $3t_1(k-\ell)-\ell = 2t_2 + t_1(k-\ell) + \ell$.

% \textcolor{red}{a connecting lemma}
\begin{lemma}\label{lm:conn}
Suppose $1\le \ell<k$ and $1/n \ll1/C \ll \beta \ll \eta\ll 1/k$ and $0<\epsilon \le \ell/(3t_1)$.
Let $H'$ and $G$ be two $n$-vertex $k$-graphs on the same vertex set $V$ such that for any $\ell$-set $S\subseteq V$, either $\deg_{H'}(S) = 0$ or $\deg_{H'}(S)\ge \eta {n}^{k-\ell}$ and $G:=\mathbb{G}^{(k)}(n, p)$ satisfies~\eqref{eq:p}.
Then for any set $W\subseteq V$ of size at most $\eta n/3$, a.a.s.~ $H'\cup G$ contains a set $\mathcal{C}$ of disjoint $t_3$-sets such that $V(\cC)\cap W=\emptyset$, $|\cC|\le \beta n$, and for every two disjoint ordered $\ell$-sets $S, S'$ in $V$ with $\deg_{H'}(S), \deg_{H'}(S')\ge \eta n^{k-\ell}$, there are at least $3\beta^3 n$ members of $\mathcal C$ that connect them.
\end{lemma}

\begin{proof}
Fix two disjoint ordered $\ell$-sets $S:=(v_1, \dots, v_{\ell})$ and $S':=(w_{\ell}, \dots, w_1)$ such that $\deg_{H'}(S), \deg_{H'}(S')\ge \eta n^{k-\ell}$.
%Since $\deg(S)\ge \eta {n}^{k-\ell}$ and the assumption on the degrees, 
We first claim that we can greedily extend $S$ to an $\ell$-path $v_1, \dots, v_{\ell+t_1(k-\ell)}$ of length $t_1$ in $H'$ such that the new vertices are disjoint from $S'\cup W$ and there are at least $(\eta/2)^{t_1} {n}^{t_1(k-\ell)}$ choices for them.
%the new vertices $(v_{\ell+1}, \dots, v_{\ell+t_1(k-\ell)})$.
Indeed, we iteratively extend the path from the current $\ell$-end $T$ by adding $k-\ell$ new vertices.
By the degree assumption, we know that $\deg_{H'}(T)\ge \eta {n}^{k-\ell}$ (in the first step $T=S$).
Since the number of $(k-\ell)$-sets that intersect the existing vertices or $W$ is $\eta n^{k-\ell}/3 + O(n^{k-\ell-1})$, there are at least $\eta {n}^{k-\ell}/2$ choices for the new $k-\ell$ vertices.

Similarly, we can greedily extend $(w_1, \dots, w_{\ell})$ to an $\ell$-path $w_1, \dots, w_{\ell+t_1(k-\ell)}$ of length $t_1$ in $H'$ such that new vertices are disjoint with $\{v_1, \dots, v_{\ell+t_1(k-\ell)}\}\cup W$ and there are at least $(\eta/2)^{t_1} {n}^{t_1(k-\ell)}$ choices for them. %the new vertices $(w_{\ell+1}, \dots, w_{\ell+t_1(k-\ell)})$. 
At last, if $t_2>0$, then we pick $t_2$ arbitrary vertices $\{u_1,\dots, u_{t_2}\}$ that are disjoint from the existing vertices and $W$, and there are at least $n^{t_2}/2$ choices for them.
Note that $t_3=2t_1(k-\ell)+t_2$.
So there are at least $(\eta /2)^{2t_1} n^{t_3}/2 \ge 24\beta n^{t_3}$ choices for the ordered $t_3$-sets
\[
(v_{\ell+1}, \dots, v_{\ell+t_1(k-\ell)}, u_1,\dots, u_{t_2}, w_{\ell+t_1(k-\ell)}, \dots, w_{\ell+1}).
\]
Let $\mathcal C_{S, S'}$ be a collection of exactly $24\beta n^{t_3}$ such ordered $t_3$-sets.
By this definition, if some $C\in \mathcal C_{S, S'}$ spans a labeled copy of $P_{t_1, t_2}$,
then $C$ connects $S$ and $S'$.
% i.e., $C$ spans a labeled copy of $P_{t_1}$ on
%\[
%(v_{t_1(k-\ell)+1}, \dots, v_{\ell+t_1(k-\ell)}, u_1,\dots, u_{t_2}, w_{\ell+t_1(k-\ell)}, \dots, w_{t_1(k-\ell)+1}),
%\]
%then $C$ connects $S$ and $S'$.
We apply Lemma~\ref{prop:prob} to $\mathcal C_{S, S'}$ for all pairs of $S, S'$ such that $\deg_{H'}(S), \deg_{H'}(S')\ge \eta n^{k-\ell}$ and $G=\mathbb{G}^{(k)}(n, p)$, and conclude that \aas~there exists a family $\mathcal C$ of disjoint $t_3$-sets such that $|\mathcal C|\le \beta n$, and for ordered $\ell$-sets $S, S'$ with $\deg_{H'}(S), \deg_{H'}(S')\ge \eta n^{k-\ell}$, there are at least
$\beta^2 n/t_3^2 \ge 3\beta^3 n$ $t_3$-sets that connect them.
In particular, $V(\cC)\cap W=\emptyset$ by our construction.
\end{proof}

%\subsection{Absorbing Lemma}
\medskip
Given a $k$-graph  $H$, let $W=\{w_1,\dots, w_{k-\ell}\}\subseteq V(H)$. The $W$-absorber is defined as follows.
Let 
\[
t_4 := \lceil (3k-\ell-2)/(k-\ell) \rceil \quad \text{and} \quad t_5:= t_4 (k-\ell) \quad (\text{thus  } 3k-\ell-2\le t_5\le 4k). 
\]
Suppose $X_i, Y_i, Z_i$, $i\in [k-\ell]$, and $T$ are pairwise disjoint ordered sets from $V(H)\setminus W$ satisfying the following properties:
\begin{enumerate}[label=($\roman*$)]
%\item $T$ and $\{X_i, Z_i, Y_i\}_{i\in [k-\ell]}$ are pairwise disjoint; \label{item:i}
\item $|X_i|=k-1$, $|Y_i|=t_5-k-i+1$, and $|Z_i|=i-1$ for every $i\in [k-\ell]$ and $|T|=\ell$; \label{item:ii}
\item $Q:=X_1 Z_2 Y_1 \ X_2 Z_3 Y_2 \cdots X_{k-\ell-1} Z_{k-\ell} Y_{k-\ell-1}  \ X_{k-\ell} Z_1 Y_{k-\ell} T$ spans a copy of $P_{t_5-1}$; \label{item:c}
\item $Q':=X_1 w_1 Z_1 Y_1 \ X_2 w_2 Z_2 Y_2 \cdots X_{k-\ell} w_{k-\ell} Z_{k-\ell} Y_{k-\ell} T$ spans a copy of $P_{t_5}$. \label{item:iv}
\end{enumerate} 
By definition, $Q$ is a $W$-absorber.
Note that $|Y_i|\ge k-1$ for $i\in [k-\ell]$ by the definition of $t_5$. Let $B_i$ be the ordered set $X_i w_i Z_i Y_i$. Since $|X_i|, |Y_i|\ge k-1$, all the edges of $Q'$ that intersect $\{w_i\}\cup Z_i$ are completely in $B_i$. Furthermore, when counting from the left end, all $X_i$ and $Y_i$ are placed at the same location in $Q$ as in $Q'$, except that $Y_{k-\ell}$ is shifted $k-\ell$ vertices to the right in $Q'$ (thus $Z_2, \dots, Z_{k-\ell}$ are simply place-holders). Consequently, if $H[B_i]\supseteq Q'[B_i]$ for $i\in [k-\ell]$ and $Q$ is a path, then $Q'$ is a path. 

The following is our absorbing lemma.
\begin{lemma}\label{lem:new}
Let $1\le \ell <k$ be integers and suppose $0<\epsilon\le \ell/(3 t_5)$ and $1/n \ll \beta \ll\eta, \alpha, 1/k,1/t_5$.
Let $V$ be a set of $n$ vertices and let $V', U$ be two (not necessarily disjoint) subsets of $V$ such that $|U|\le \eta n/3$.
Let $H$ be a $k$-graph on $V$ such that $\deg_H(v)\ge \alpha n^{k-1}$ for all $v\in V'$, and for all $\ell$-sets $S\subseteq V$, either $\deg_{H}(S)=0$ or $\deg_{H}(S) \ge \eta n^{k-\ell}$.
%Let $A$ be the $(k-\ell)$-graph on $H$ whose edges are the $(k-\ell)$-tuples $S\subseteq V$ such that there are at most $\delta n^{\ell + (k-1) (k-\ell)}$ ordered copies $Q$ of $P_k$ satisfying that $S_Q=S$.
Suppose $G:=\mathbb{G}^{(k)}(n, p)$ has vertex set $V$ and satisfies \eqref{eq:p}. 
Then $H\cup G$~\aas~contains a family $\cA$ of at most $\beta n$ vertex-disjoint copies of $P_{t_5-1}$ with ends in $\partial_\ell H$ such that $V(\cA) \subseteq V\setminus U$, and every $(k-\ell)$-set $W\subseteq V'$ has at least $\beta^3 n$ $W$-absorbers in $\cA$.
\end{lemma}

\begin{proof}
%[Proof of Lemma~\ref{lem:new}]
%Choose a constant $\mu$ such that $\beta \ll \mu\ll \eta, \alpha, 1/k$.
For each $W=\{w_1,\dots, w_{k-\ell}\} \subseteq V'$, we will find $W$-absorbers from $V\setminus U$ satisfying~\ref{item:ii} -- ~\ref{item:iv}.
We achieve this in two steps.
In the first step, for each $i\in [k-\ell]$, we will find a path $Q_i$ of length $t_4$ with $V(Q_i)=  \{v_1, \dots, v_{t_5+\ell} \} \subseteq V\setminus U$ such that $v_k = w_i$, and there are at least $\frac{\alpha}2 (\frac{\eta}2)^{t_4 -1} n^{t_5+\ell-1}$ choices for $V(Q_i)$. Indeed, we first choose an unordered set $\{v_1, \dots, v_{k-1}\} \in N_H(w_i)$. Since $\deg_H(w_i)\ge \alpha n^{k-1}$ and at most $|U| + \ell+ t_5(k-\ell) \le \eta n/2$ vertices are either in $U$ or used in this step, there are at least $\frac{\alpha}2  n^{k-1}$ choices for $\{v_1, \dots, v_{k-1}\}$. Next, let $S=\{v_{k- \ell+1}, \dots, v_k\}$. Since $\deg_{H}(S)>0$, we have $\deg_{H}(S) \ge \eta n^{k-\ell}$. Hence we can choose an unordered set $\{v_{k+1}, \dots, v_{2k-\ell}\} \in N_H(S)$ while avoiding $U$ and the vertices already used in this step.
There are at least $\frac{\eta}2  n^{k-\ell}$ choices. 
We repeat this to obtain the desired path $Q_i$ and there are at least $\frac{\alpha}2 (\frac{\eta}2)^{t_4 -1} n^{t_5+\ell-1}$ choices for $V(Q_i)$ as an  ordered set.
Let $B_i$ be the ordered set $\{v_1, \dots, v_{t_5} \}$. It follows that there are at least $\frac{\alpha}2 (\frac{\eta}2)^{t_4 -1} n^{t_5-1}$ choices for $B_i$. Now let $A= B_1 \cdots B_{k-\ell-1} V(Q_{k-\ell})$. We have at least $((\frac{\alpha}2) (\frac{\eta}2)^{t_4-1})^{k-\ell} n^{\ell+(t_5-1)(k-\ell)} \ge 24 \beta n^{\ell+(t_5-1)(k-\ell)}$ choices for~$A$.

Now we proceed to the second step. For each $ i\in [k-\ell]$, recall that $V(Q_i)=  \{v_1, \dots, v_{t_5+\ell} \}$. Define (ordered) sets
\[
X_i= \{v_1, \dots, v_{k-1} \},\quad Z_i=\{v_{k+1}, \dots, v_{k+i-1} \}, \quad \text{and} \quad Y_i = \{v_{k+i}, \dots, v_{t_5} \}.
\]
In addition, let $T=\{v_{t_5+1}, \dots, v_{t_5+\ell} \}$ from $Q_{k-\ell}$.
It is clear that $X_i, Y_i, Z_i$ and $T$ satisfy~\ref{item:ii}.
Recall that $B_i = X_i w_i Z_i Y_i$. For $Q'$ defined in \ref{item:iv}, our first step already provides the edges of $Q'[B_i]$ for $i\le k-\ell-1$ and the edges of $Q'[B_{k-\ell}\cup T]$. 
Following the discussion right after \ref{item:iv}, we achieve both \ref{item:c} and \ref{item:iv} if $Q$ is a path. To this end, we use the edges of $G$. 
Let $\mathcal F_W$ be the family of $24 \beta n^{\ell + (t_5-1) (k-\ell)}$ copies of $A$, each re-ordered as in $Q$.
We apply Lemma~\ref{prop:prob} to $G$ with $x=0$, families $\mathcal F_{W}$ for all ordered $(k-\ell)$-sets $W\subseteq V'$, and conclude that \aas~there exists a collection $\cA$ of at most $\beta n$ vertex-disjoint copies of $P_{t_5-1}$ such that for every $(k-\ell$)-set $W\subseteq V'$, at least $\beta^3 n$ members of $\cA$ are from $\mathcal F_{W}$, and thus are
$W$-absorbers.
%Together with the edges of $H$ we found in the first step, $A=U\cup \bigcup_{i\in [k-\ell]}(X_i\cup Z_i\cup Y_i)$ also satisfies~\ref{item:iv}. So each such $Q$ is a $W$-absorber.
At last, because of the first step, both $\ell$-ends of these paths are in $\partial_{\ell} H$.
%Note that by the definition of $Q$, $X_1$ in $Q_1$ spans an edge, and 
\end{proof}

In the proof of Theorem~\ref{main} we need a lemma to cover most of the vertices with constantly many paths.
This is done in the following lemma.
In the proofs of the following lemma and Theorem~\ref{main}, we use the trick of \emph{multi-round exposure}, namely, in each of the steps later, we expose one or several independent copies of the binomial random hypergraph, each of them with edge probability a constant fraction of the original edge probability.
%Given a $k$-graph $H$, a vertex $v$ and a vertex set $W$, define $N_{H}(v, W) = N_{H}(v)\cap \binom{W}{k-1}$.

%\begin{lemma}\label{lm:almost}
%Let $1\le \ell < k$, and suppose $1/n \ll 1/C \ll \zeta\ll \alpha, 1/k$ and $0<\epsilon \le \ell/(4\zeta^3)$.
%Suppose $V$ is a set of $n$ vertices that is partitioned into $V'\cup R$ with $|R|\ge \zeta n/2$.
%Let $G:=\mathbb{G}^{(k)}(n, p)$ on $V$ satisfying~\eqref{eq:p}.
%Let $L$ be an $\ell$-graph on $R$ with at least $\alpha |R|^\ell$ edges when $\ell\ge 2$, and $L=R$ when $\ell=1$.
%Moreover, when $\ell=1$, there is also a $k$-graph $H$ on $V$ such that $|N_H(v, R)|\ge \alpha |R|^{k-1}$ for every $v\in V$; when $\ell\ge 2$, let $H$ be the empty $k$-graph on $V$.
%Then \aas~$G\cup H$ contains a set of at most $2\zeta^3 n$ vertex-disjoint $\ell$-paths whose ends are in $L$ and which cover $V'$.
%\end{lemma}
\begin{lemma}\label{lm:almost}
Let $1\le \ell < k$, and suppose $1/n \ll 1/C \ll \zeta\ll \alpha\ll 1/k$ and $0<\epsilon \le \zeta^3\ell/6$.
Suppose $V$ is a set of $n$ vertices and $V_0\subseteq V$ with $|V_0|\le \alpha n$, and furthermore, when $\ell=1$, suppose that $V_0=\emptyset$.
%In particular, $V_0=\emptyset$ when $\ell=1$.
Suppose $G:=\mathbb{G}^{(k)}(n, p)$ on $V$ satisfying~\eqref{eq:p}.
%Let $L$ be an $\ell$-graph on $V\setminus V_0$ with $|L|\ge \alpha n^\ell$ when $\ell\ge 2$, and $L=V$ when $\ell=1$.
Let $L$ be an $\ell$-graph on $V\setminus V_0$ with $|E(L)|\ge \alpha n^\ell$.
%Moreover, when $\ell=1$, there is also a $k$-graph $H$ on $V$ such that $|N_H(v, R)|\ge \alpha |R|^{k-1}$ for every $v\in V$; when $\ell\ge 2$, let $H$ be the empty $k$-graph on $V$.
Then \aas~$G$ contains a set $\mathcal P$ of at most $2\zeta^3 n$ vertex-disjoint $\ell$-paths such that their ends are in $L$, $V_0\subseteq V(\mathcal P)$ and $|V\setminus V(\mathcal P)|\le 2\zeta n$.
\end{lemma}

\begin{proof}
%As $V_0=\emptyset$ and $L=V$ when $\ell=1$, it suffices to find a set of at most $2\zeta^3 n$ vertex-disjoint $\ell$-paths.
%This can be done by an application of Lemma~\ref{lm:gnp}~\eqref{item-I}.

Since $|L|\ge \alpha n^\ell$,
% (in fact, $|L|=|V|=n$ for $\ell=1$), 
by averaging, there exists a set $R\subseteq V\setminus V_0$ of size $\zeta n$ such that $ | L \cap \binom{R}{\ell} | \ge \alpha |R|^\ell /2$.
%we can apply Lemma~\ref{prop:sample} to $V$ with $q=\beta'=\zeta$ and $\mathscr F= \{L\}$ obtaining a set $R\subseteq V\setminus V_0$ such that $(1-\zeta)\zeta n \le |R|\le (1+\zeta)\zeta n$ and $ |L \cap \binom{R}{\ell}| \ge \zeta^\ell (1-\zeta) |L|$. 
%Using $|L|\ge \alpha n^\ell$ and $|R|\le (1+\zeta)\zeta n$, we derive that 
%\begin{equation}\label{item:a}
%$ | L \cap \binom{R}{\ell} | \ge \alpha \zeta^\ell(1-\zeta)n^\ell $.
%\end{equation}
%When $\ell=1$, since $\delta_1(H)\ge \alpha n^{k-1}$, we can apply Lemma~\ref{prop:sample} to $V$ with $q=\beta'=\zeta$, and $\mathscr F= \{ N_{H}(v): v\in V\}$ obtaining a set $R\subseteq V$ such that $(1-\zeta)\zeta n \le |R|\le (1+\zeta)\zeta n$ and for each $v\in V$, 
%\begin{equation}\label{item:b}
%|N_{H}(v, R)| \ge \zeta^{k-1} (1-\zeta)|N_{H}(v)|\ge  \alpha |R|^{k-1}/8,
%\end{equation}
%where the second inequality follows from $|N_{H}(v)|\ge \alpha n^{k-1}$ and $|R|\le (1+\zeta)\zeta n$.

%We cannot put the vertices of $V_0$ into all different paths because it is possible to have $|V_0| > 2\zeta^3 n$.
We find our path cover in two phases.
In the first phase, we use relatively long paths with ends from $R$ to cover most of the vertices of $V$. In the second phase, we greedily cover the remaining vertices of $V_0\setminus R$ with short paths.
We therefore expose $G$ in two rounds such that $G=G_1\cup G_2$, where each $G_i$ is $\mathbb{G}^{(k)}(n,p')$ with $(1-p')^{2} = 1-p$. Thus $p' >p/2 > n^{-(k-\ell)-2\epsilon}$ when $\ell\ge 2$.

We start with Phase 1. Let $s$ be the smallest integer such that $s\ge 1/\zeta^3$ and $s\equiv \ell \mod{(k-\ell)}$, and let $s_1=(s-\ell)/(k-\ell)$. Since $\epsilon \le \zeta^3 \ell/3$, we have $2\eps \le \ell / (3s_1)$. 
By Lemma~\ref{lm:gnp}~\eqref{item-I}, \aas~for all $V^*\subseteq V\setminus R, R'\subseteq R$ satisfying $|V^*|\ge \zeta^3 n$, $|R'|\ge |R|/2$ and $|L\cap \binom{R'}{\ell}| \ge (\alpha/3) |R|^\ell$,
$G_1=\mathbb{G}^{(k)}(n,p')$ contains a copy of $P_{s_1}$ whose $\ell$-ends are in $L\cap \binom{R'}{\ell}$ and other vertices are from $V^*$.
Owing to this property, we repeatedly construct copies of $P_{s_1}$ by letting $V^*$ be the set of uncovered vertices of $V'$ and letting $R'$ be the set of uncovered vertices of $R$, as long as $|V^*|\ge \zeta^3 n$.
This is possible because we construct at most $\zeta^3 n$
vertex-disjoint copies of $P_{s_1}$, which consume at most $2\ell \zeta^3 n$ vertices from $R$.
During the process, at least $|R| - 2\ell \zeta^3 n\ge |R|/2$ vertices of $R$ are available and by our assumption, they span at least $\alpha |R|^\ell /2   - 2\ell \zeta^3 n\cdot |R|^{\ell-1} \ge \alpha |R|^\ell /3$ edges of $L$.
Let $\mathcal{P}_{1}$ denote the set of the paths obtained in this phase.

Note that when $\ell=1$, since $V_0=\emptyset$ and $|V\setminus V(\mathcal P_1)|\le |R|+\zeta^3 n\le 2\zeta n$, we are done by letting $\mathcal P=\mathcal P_1$.

Now we proceed to Phase 2 and assume that $\ell\ge 2$.
Let $V''$ be the set of uncovered vertices in $V\setminus R$ and $R'=R\setminus V(\mathcal{P}_{1})$.
Note that $|V''|\le \zeta^3 n$ and $|R'|\ge |R| - 2\ell \zeta^3 n \ge |R|/2$, and $|L\cap \binom{R'}{\ell}| \ge (\alpha/3) |R|^\ell$.
Using the edges of $G_2=\mathbb{G}^{(k)}(n,p')$, we will greedily put vertices $v\in V''$ into vertex-disjoint $\ell$-paths $w_1 \cdots w_{k-1} v w_{k} \cdots w_{t_6(k-\ell)+\ell-1}$ of length $t_6:=\lfloor ({k-1})/({k-\ell}) \rfloor+1$ such that all the vertices other than $v$ are from $R'$ and both $\ell$-ends are in $L$.
Note that $v$ is in every edge of the path but in neither of the $\ell$-ends.
%Below we explain how to obtain such paths for the cases $\ell=1$ and $\ell\ge 2$ separately.

%First assume $\ell=1$.
%In this case we only use the edges of $H$.
%Note that $t_4=2$.
%We put all vertices $v\in V''$ into vertex-disjoint $1$-paths of length $2$ such that $v$ has degree $2$.
%%This can be done because of~(b) and the fact that $R$ is much larger than $V''$.
%Indeed, by the degree condition of $H$, for each vertex $v\in V''$, we have $|N_{H_*}(v, R')| \ge \alpha |R|^{k-1} - 2\ell \zeta^3 n |R|^{k-2} \ge \alpha |R|^{k-1}/2$. Since $|R|\ge \zeta n/2$ and $|V''|\le \zeta^3 n$, there are at least 
%\[
%\frac{\alpha  |R|^{k-1}}2 \left( \frac{\alpha  |R|^{k-1}}2 - k |R'|^{k-2} \right) \ge \frac{\alpha^3 |R|^{2k-2}}{6}
%%- |R\setminus R'| |R|^{2k-3} \ge \frac{\alpha^3 |R'|^{2k-2}}{65}
%\ge \frac{\alpha^3 \zeta n}{12} |R|^{2k-3} > (2k-2)|V''| |R'|^{2k-3},
%\]
%(unlabeled) $1$-paths of length $2$ containing $v$ and $2k-2$ vertices from $R'$.
%As a result, we can construct the desired $1$-paths greedily because the process consumes at most $(2k-2)|V''|$ vertices from $R'$, which makes at most $(2k-2)|V''| |R'|^{2k-3}$ $1$-paths unavailable.

%Now assume $\ell\ge 2$ and we use the edges of $G_2=\mathbb{G}^{(k)}(n,p')$ in this case.
For any $v\in V''$,  let $G_v$ be the edges of $G_2$ that contain $v$ and have their other $k-1$ vertices from $R'$. 
%So $H_v$ is isomorphic to $\mathbb{G}^{(k-1)}(n-|V_{0}''|,p')$, the binomial random $(k-1)$-graph on $V\setminus V_{0}''$.
%For distinct vertices $u, v\in V''$, edges in $G_u$ and $G_v$ are independent from each other.
For distinct vertices $u, v\in V''$, the possible edges appear in $G_u$ independently of the possible edges that can appear in $G_v$.
Suppose we consider $v\in V''$ after covering some vertices of $V''$ by $\ell$-paths.
To this end, we expose $G_v$.
Let $R''$ be the set of unused vertices in $R'$.
We have $|R''|\ge |R'| - |V''| (t_6(k-\ell)+\ell)\ge |R'| - 2k\zeta^3 n \ge |R|/3$ and $|L\cap \binom{R''}{\ell}| \ge |L\cap \binom{R'}{\ell}| - 2k\zeta^3 n |R'|^{\ell-1} \ge (\alpha/4) |R|^\ell$.
We choose two disjoint $\ell$-sets from $L\cap \binom{R''}{\ell}$ and $t_6(k-\ell) - \ell-1$ vertices from $R''$ forming an ordered $(t_6(k-\ell)+\ell - 1)$-set $(w_1, \dots, w_{t_6(k-\ell)+\ell-1})$ -- there are 
\[
\frac{\alpha |R|^\ell}{4}\cdot \left(\frac{\alpha |R|^\ell}{4} - \ell |R''|^{\ell-1} \right) \cdot 
\left(\frac{\zeta n}{4}\right)^{t_6(k-\ell)-\ell -1} \ge \alpha^3 (\zeta n)^{t_6(k-\ell)+\ell - 1}
\]
such sets. 
We observe that $w_1 \dots w_{k-1} v w_{k} \dots w_{t_6(k-\ell)+\ell-1}$ spanning a copy of $P_{t_6}$  in $G_v$ is equivalent to $w_1 \dots w_{t_6(k-\ell)+\ell-1}$ spanning a $(k-1)$-uniform $(\ell-1)$-path in $N_{G_v}(v)$.
Since $p'\ge n^{(k-1)-(\ell-1) - 2\epsilon}$ and $2\epsilon \le (\ell-1)/(3 t_6)$, we can apply Lemma~\ref{lm:gnp2} to $\alpha^3 (\zeta n)^{t_6(k-\ell)+\ell - 1}$ ordered $(t_6(k-\ell)+\ell - 1)$-sets, and conclude that $G_v$ contains a desired $\ell$-path with probability at least $1-\exp(-n^{1/3})$. 
By the union bound, with probability at least $1-|V''|\exp(-n^{1/3}) =1-o(1)$, we can put all the vertices of $V''$ into vertex-disjoint $\ell$-paths of length $t_6$ by using the vertices of $R$ such that all the $\ell$-ends are in $L$.
This finishes Phase 2.
%The proof is completed.
Let $\mathcal{P}_{2}$ denote the family of the $\ell$-paths found in this phase.
Let $\mathcal{P}:=\mathcal{P}_1\cup \mathcal{P}_{2}$ and note that $|\mathcal P|\le 2\zeta^3 n$.
By construction, all the $\ell$-ends of the paths in $\mathcal P$ are in $L$.
Since $V\setminus V(\mathcal P)\subseteq R$, we have $V_0\subseteq V(\mathcal P)$ and $|V\setminus V(\mathcal P) |\le |R|\le 2\zeta n$.
\end{proof}

\section{Proof of Theorem~\ref{main}}

%In this section we prove Theorem~\ref{main}.
In this section we prove Theorem~\ref{main}.  We essentially follow the procedure mentioned in Section~1.3 but need additional work.
%We first discuss main ideas in the proof.
%In addition to the three steps mentioned in Section~1.3, we need three pre-processing steps. In Step 1, we apply Lemma~\ref{lem:shave} twice and obtain a subgraph $H_*$ of $H$ on $V_*\subset V$ such that every $\ell$-set in $L:=\partial_{\ell} H_*$ has a large degree in $H_*$ and $V\setminus V_*$ is very small. 
%In Step 2, we apply Lemma~\ref{prop:sample} and obtain a small \emph{reservior} set $R\subseteq V_*$ such that a proportional number of $\ell$-sets of $R$ are in $L$. The role of $R$ is to provide $\ell$-ends for later steps.
%In Step 3 we select two sets of connectors $\cC_1, \cC_2$ by Lemma~\ref{lm:conn}. The (larger) connector $\cC_2$ will be used to connect the absorbers $\cA$ provided by Lemma~\ref{lem:new} while $\cC_1$ will be used to connect the paths provided by Lemma~\ref{lm:almost}. Note that after connecting $\cA$, the remaining members of $\cC_2$ will be discarded (and the vertices in these members need to be covered in Step 5). 
%In Step 6, the vertices that are absorbed are from $V(\cC_1)\cup R$. 
%By Lemma~\ref{lem:new}, this requires $(V(\cC_1)\cup R) \subseteq V_*$.
%In what follows, when we say we apply a lemma to a graph $H$, it means that we apply it to the union of $H$ and a random $k$-graph $\mathbb{G}^{(k)}(|V(H)|,p)$.
%This means that we expose one copy of $\mathbb{G}^{(k)}(n,p)$ in most steps, which can be done by the multi-round exposure trick.
We first apply Lemma~\ref{lem:shave} and obtain a spanning subgraph $H'$ of $H$. Let $V_*$ be the set of vertices of $H'$ with high degree.
Following the procedure outlined in Section~1.3, we obtain an absorbing path $P_{abs}$, a set $\cC_1$ of connectors and a set $\cP$ of paths that cover almost all the vertices.
A natural attempt is to use the connectors in $\cC_1$ to connect the paths in $\cP$ and $P_{abs}$ to obtain an almost spanning cycle and then absorb the remaining vertices of $\cC_1$ by $P_{abs}$.
On the other hand, when applying Lemma~\ref{lem:new} to $H'$, we can only absorb vertices in $V_*$. Therefore we need to have $V(\cC_1)\subseteq V_*$.
However, we cannot strengthen Lemma~\ref{lm:conn} by asking $V(\cC_1)\subseteq V_*$ because for a given $\ell$-set in $V_*$, it is possible that all its neighbors intersect $V\setminus V_*$ (recall that $\deg_{H'}(S)\ge \eta^2 n^{k-\ell}$ and $|V\setminus V_*|\le \eta n$). Therefore, this naive attempt fails. 
%and we need a stronger control on the $\ell$-sets to be connected.

To fix it, we ``shave'' $H'$ again, namely, applying Lemma~\ref{lem:shave} to $H'[V_*]$, and obtain a spanning $k$-graph $H_*$ on $V_*$. We thus apply Lemma~\ref{lm:conn} to $H_*$ and obtain $\cC_1$ such that $V(\cC_1)\subseteq V_*$ and $\cC_1$ can connect any two $\ell$-sets in $L:=\partial_{\ell} H_*$. 
In order to obtain $P_{abs}$, we apply Lemma~\ref{lem:new} to $H'$ obtaining a family $\cA$ of absorbers and apply Lemma~\ref{lm:conn} to $H'$ obtaining another set $\cC_2$ of connectors. After connecting $\cA$ into $P_{abs}$, unused members of $\cC_2$ will be discarded (and the vertices in these members will be covered in a later step).

Below are the details of our proof.
%\subsection{Proof of Theorem~\ref{main}}
Let $1/n\ll 1/C \ll \zeta \ll \beta \ll \eta \ll\alpha, 1/k$  and $0<\epsilon \le \zeta^3\ell/12$.
%Set $r:=\lfloor 1/\alpha \rfloor +1$ and
%Assume $p=p(n)$ satisfy~\eqref{eq:p1}.
Write $V=V(H)$.
Let $\bigcup_{i\in [4]} G_i = \mathbb{G}^{(k)}(n,p)$ such that each $G_i$ is $\mathbb{G}^{(k)}(n,p')$ and $(1-p')^{4} = 1-p$.
In particular, $p' >p/4 > n^{-(k-\ell)-2\epsilon}$ if $\ell\ge 2$, and $p' >p/4 \ge (C/4) n^{-(k-1)}$ if $\ell=1$.
%So when we apply other auxiliary results, we apply them with $\omega /(2^r k+3)$ in place of $\omega$.
%So when we apply Lemmas~\ref{lm:gnp},~\ref{prop:prob},~\ref{lm:conn} and~\ref{lem:breakparity} later, if $\ell\ge 2$, then we apply them with $2\eps$ in place of $\eps$; if $\ell=1$, then we apply them with $C /5$ in place of $C$.
When we apply Lemmas~\ref{lm:conn}, ~\ref{lem:new} and~\ref{lm:almost}, we apply them with $p'$ in place of $p$ and $2\epsilon$ in place of $\epsilon$.

\medskip
\noindent\textit{Step 1. Shave $H$ twice.}
We define two subgraphs $H'$ and $H_*$ of $H$ as follows.
If $\ell=1$, let $H'=H_* =H$, $V_*= V$,  and $L=V$.
If $\ell\ge 2$, then we apply Lemma~\ref{lem:shave} to $H$ and obtain a subgraph $H'$ with the following properties:
\begin{itemize}
%\item $e(H')\ge {\alpha} n^{k}/(2k)$.
\item there exists $V_0\subseteq V$ such that $|V_0| \le 3k\eta^2 n/\alpha\le \eta n$ and $\deg_{H'}(v)\ge 2\alpha n^{k-1}/3$ for all $v\in V\setminus V_0$;
\item for every $\ell$-set $S\subseteq V$, either $\deg_{H'}(S)=0$ or $\deg_{H'}(S) \ge \eta^2 n^{k-\ell}$.
\end{itemize}
%Let $L$ be the collection of $\ell$-sets $S$ in $V$ such that $\deg_{H'}(S)\ge \eta^2 n^{k-\ell}$.
%If $\ell=1$, then let $H'=H$.
%Define $L$ and $A$ as at the beginning of this section.
%Note that $\binom k{\ell}e(H')\le |L| {n}^{k-\ell} + {n}^{\ell}\cdot \eta^2 n^{k-\ell}$, which implies that $|L|\ge \alpha n^{\ell}/3 - \eta^2 n^\ell \ge \alpha n^\ell/4$.
%
%We apply Lemma~\ref{lem:reg} on $H'$ and get a regular partition $V(H')=V_0\cup V_1\cup V_2\cup \cdots \cup V_t$ for some $t\le M$ such that each part has size at least $n/(2t)$ and $|V_0|\le \beta n$.
%We next apply Lemma~\ref{prop:extension} to $H'$ (with $\eta$ in place of $\beta$) and obtain a partition $V_0\cup V_1\cup V_2\cup \cdots \cup V_m$ of $V$, such that $m\le t$, $|V_{0}|\le 2\sqrt \eta n$ and $|V_i|\ge n/(2m)$ for $i\in [m]$, 
%Moreover, for every $v\in V_i$, $\deg_{A[V_i]}(v) \le  2\sqrt\eta {|V_i|}^{k-\ell-1}$.
%
%Let $H_*=H'[V\setminus V_0]$.
%Let $V_0$ denote the set of vertices with degree lower than $2\alpha n^{k-1}/3$ in $H'$. So $|V_0|\le \eta n$.
Let $V_*=V\setminus V_0$ and $n_*:=|V_*| \ge (1-\eta )n$.
We have $\delta_1(H'[V_*]) \ge 2\alpha n^{k-1}/3 - |V_0| n^{k-2} \ge \alpha n_*^{k-1}/2$.
Apply Lemma~\ref{lem:shave} again to $H'[V_*]$ and obtain a subgraph $H_*$ on $V_*$ such that
\begin{itemize}
\item $e(H_*)\ge \alpha n_*^{k}/(4k) $,
%\item $\deg_{H_*}(v)\ge \alpha n_*^{k-1}/3$ for all but at most $3k\eta^2 n_*/(\alpha/2)\le \eta n$ vertices of $V_*$.
\item for every $\ell$-set $S\subseteq V_*$, either $\deg_{H_*}(S)=0$ or $\deg_{H_*}(S) \ge \eta^2 n_*^{k-\ell}$.
\end{itemize}
%If $\ell=1$, then let $H_* = H[V_*]$.
%Let $L$ be the collection of $\ell$-sets $S$ in $V$ such that $\deg_{H'}(S)\ge \eta^2 n^{k-\ell}$.
Let $L=\partial_\ell H_*$. We have 
%Note that by definition, $L_*\subseteq L$.
%Moreover, if $\ell=1$, then $L_*=V$.
%Note that if $\ell=1$, then $H'=H$, $L=V$ and $\delta_{1}(H')\ge \alpha {n}^{k-1}$.
%Note that $\binom k{\ell}e(H_*)\le |L| n_*^{k-\ell} + n_*^{\ell}\cdot \eta^2 n_*^{k-\ell}$, which implies that
\begin{align}\label{eq:L}
|L|\ge \frac{\binom k{\ell} e(H_*)}{\binom {n-\ell}{k-\ell} }\ge \frac{ \binom k{\ell} \frac{\alpha}{4k} n_*^k}{n_*^{k-\ell}}  \ge \frac{\alpha}4 n_*^{\ell}.
\end{align}

\medskip
\noindent\textit{Step 2. Build connectors $\cC_1$ and $\cC_2$.}
%Now we select the connectors.
%Consider $H''=H'[V'']$ and $H_{2^r k+2}[V''] = \mathbb G^{(k)}(n'',p')$. 
%Since $|V(\mathcal{F})|+|V(\mathcal P)|\le  \sqrt\zeta n$, we have
%$\delta_{1}(H'')\ge \alpha n^{k-1} - \sqrt\zeta n\cdot n^{k-2} \ge \alpha n^{k-1}/2$. 
%Recall that $V^1=V\setminus (B\cup R)$ and $V_i^1 = V_i\cap V^1$.
We obtain $\cC_1$ and $\cC_2$ by applying Lemma~\ref{lm:conn} twice. %independently.
%Note that $|B|+|R|\le 3\sqrt\zeta n\le \eta^2 n/3$.
First, we apply Lemma~\ref{lm:conn} to $H_*\cup G_{1}[V_*]$ with $W=\emptyset$, $\eta^2$ (in place of $\eta$) and $\zeta$ (in place of $\beta$), and conclude that $H_*\cup G_{1}[V_*]$ \aas~contains a set $\cC_1$ of disjoint $t_3$-sets such that $V(\cC_1)\subseteq V_*$, $|\cC_1|\leq \zeta n$ and for any two disjoint ordered $\ell$-sets $S, S'$ in $L$, there are at least $3\zeta^{3} n$ members of $\cC_1$ connecting them. 
Second, we apply Lemma~\ref{lm:conn} to $H'\cup G_{2}$ with $W=V(\cC_1)$, $\eta^2$ (in place of $\eta$) and $\beta$, and conclude that $H'\cup G_{2}$ \aas~contains a set $\cC_2$ of disjoint $t_3$-sets such that $V(\cC_2)\subseteq V\setminus V(\cC_1)$, $|\cC_2|\leq \beta n$, and for any two disjoint ordered $\ell$-sets $S, S'$ in $\partial_\ell H'$, there are at least $3\beta^3 n$ members of $\cC_2$ connecting  them. 
%Next note that $H_*[V_*\setminus B]$ satisfies that any $\ell$-set $S$ either has degree $0$ or has degree at least $\eta^2 n_*^{k-\ell} - |B| n_*^{k-\ell-1} \ge (\eta^2/2)n_*^{k-\ell}$.

%Then let $\cC_2$ be the subfamily of $\cC_2'$ obtained by deleting the members that intersect $B\cup R$, and let $\cC_1$ be the subfamily of $\cC_1'$ obtained by deleting the members that intersect $B\cup R\cup V(\cC_2)$.
%Thus, we have $|\cC_1|\leq \beta n$, $|\cC_2|\leq \zeta^{1/8} n$, and 
%\begin{itemize}
%\item for any two disjoint ordered $\ell$-sets $S, S'$ in $L\cup L_*$, there are at least $2\beta^3 n - |B\cup R\cup V(\cC_2)| \ge \beta^3 n+2$ members of $\cC_1$ connecting them;
%\item for any two disjoint ordered $\ell$-sets $S, S'$ in $L_*$, there are at least $2\zeta^{3/8} n - |B\cup R| \ge \zeta^{3/8} n$ members of $\cC_2$ connecting them.
%\end{itemize}

\medskip
\noindent\textit{Step 3. Build an absorbing path.}
Note that $|V(\cC_1\cup \cC_2)| \le 2\beta n \cdot t_3< 6k\beta n$ (as $t_3< 3k$).
We apply Lemma~\ref{lem:new} to $H'\cup G_3$ with $V'= V_*$, $U= V(\cC_1\cup \cC_2)$, 
$2\alpha/3$ (in place of $\alpha$), $\eta^2$ (in place of $\eta$) and $\beta^3$ (in place of $\beta$). 
%Fix an ordered $(k-\ell)$-set $S$. 
%By Lemma~\ref{lem:new}, there are at least $\delta n^{\ell + (a-1) (k-\ell)}$ ordered $(\ell + (a-1) (k-\ell))$-sets $Q$ of $P_k$ in $H'$ such that $S_Q=S$ and their ends are in $L$. 
%Since these paths are in $H'$, their $\ell$-ends are in $L$.
%Moreover, since $|V(\cC_1\cup \cC_2)|+|R|\le 2\beta n \cdot t_3+2\zeta n\le 6k\beta n$, 
%at least $\delta n^{\ell + (a-1) (k-\ell)} - 6k\beta n^{\ell + (a-1) (k-\ell)} \ge 24\beta^3 n^{\ell + (a-1) (k-\ell)}$ 
%such ordered sets $Q$ satisfy $(V(Q)\setminus S) \cap (V(\cC_1\cup \cC_2)\cup R)= \emptyset$.
%So by exposing $G_3=\mathbb{G}^{(k)}(n,p')$, Lemma~\ref{lem:new} gives that 
Then \aas~there exists a collection $\cA$ of at most $\beta^3 n$ vertex-disjoint copies of $P_{t_5 -1}$ such that for every $(k-\ell)$-set $S\subseteq V_*$, there are at least $\beta^9 n$ $S$-absorbers in $\cA$. Note that each member of $\cA$ contains $(t_5 -1)(k- \ell) + \ell \le 4k^2$ vertices.
Moreover, all the members of $\cA$ have their $\ell$-ends in $ \partial_\ell H'$ and $V(\cA)\subseteq V\setminus V(\cC_1\cup \cC_2)$.
Next, we pick two disjoint $\ell$-sets $E_1, E_2\in L$, which are also disjoint from $V(\cA)\cup V(\cC_1\cup \cC_2)$.
This is possible because $|V(\cA)\cup V(\cC_1\cup \cC_2)|\le 4 k^2 \beta^3 n + 6k\beta n\le 7k\beta n$ and $|L|\ge \alpha n_*^\ell/4$.
Finally, we use the members of $\cC_2$ to connect the members of $\cA$ and $E_1, E_2$ to an $\ell$-path $P_{abs}$ with ends $E_1$ and $E_2$, which is possible because all the absorbers have ends in $\partial_\ell H'$ and $|\cA|\le \beta^3 n$.

\medskip
\noindent\textit{Step 4. Cover most of the remaining vertices.}
Let $V'=V\setminus (V(P_{abs})\cup V(\cC_1))$. 
Note that $|V(P_{abs})\cup V(\cC_1)|\le 7k\beta n + 2\ell  \le 8k\beta n$. %|V(\cA)\cup V(\cC_1\cup \cC_2)\cup E_1\cup E_2|
Let $L':=L[V_*\setminus (V(P_{abs})\cup V(\cC_1))]$. %\setminus V_0
By~\eqref{eq:L}, we have
\[
|L'| \ge \alpha n_*^\ell/4 - |V(P_{abs})\cup V(\cC_1)|\cdot n_*^{\ell-1} \ge \alpha n_*^\ell/5 \ge \alpha |V'|^\ell/6.
\]
So we can apply Lemma~\ref{lm:almost} with $V'$ (in place of $V$), $V_0$, $L'$ (in place of $L$), $\alpha/6$ (in place of $\alpha$), $G=G_4$, and \aas~obtain a collection $\cP$ of at most $2\zeta^3 n$ vertex-disjoint paths with ends in $L$, which leaves a set $W$ of at most $2\zeta n$ vertices in $V'\setminus V_0\subseteq V_*$ uncovered.
Next, we connect $P_{abs}$ and the paths in $\mathcal{P}$ by the connectors in $\cC_1$ and denote the resulting $\ell$-cycle by $Q$.
This is possible because the ends of these paths are in $L$, and $1 + |\mathcal{P}|\le 1 + 2\zeta^3 n \le 3\zeta^3 n$.

\medskip
\noindent\textit{Step 5. Finish the Hamiltonian $\ell$-cycle.}
Let $X=V\setminus V(Q)$. The construction of $Q$ implies that $|X|\in (k-\ell)\mathbb N$, $X \subseteq W\cup V(\cC_1)\subseteq V_*$ and $|X|\le 2\zeta n+t_3 \zeta n \le 2 t_3 \zeta n$ (because $t_3\ge 2\ell\ge 2$).
%Since $k-\ell$ divides both $n$ and $|V(C)|$, we have $|B\cup X|\in (k-\ell)\mathbb{N}$.
%By the property of $\mathcal{P}_1$, we can remove a set $Y$ of at most $k m$ vertices from $\cP_1$ such that $|(B\cup X\cup Y)\cap V_i|\in (k-\ell)\mathbb{N}$ for all $i\in [m]$, and $V(C)\setminus Y$ still spans an $\ell$-cycle, denoted by $C'$.
%In particular, $C'$ contains the $\ell$-paths in $\cA$ as subpaths.
%If $\ell=k-1$, then as noted after Lemma~\ref{prop:extension}, since $X\cap V_0=\emptyset$, every vertex in $X$ is not in $A$.
%In this case let $M=X$.
%Otherwise, fix $i\in [m]$ and let $W_i:=(B\cup X\cup Y)\cap V_i$.
%Since $\zeta \ll 1/t$ and $|V_i|\ge n/(2t)$, we have $|X|\le 3t_3 \zeta n \le 6t_3 t \zeta |V_i| \le \zeta^{2/3} |V_i|$.
%Recall that $|B_i|\ge \sqrt\zeta |V_i|/2$. 
%It follows that $|W_i\setminus B_i|\le \zeta^{2/3} |V_i| + k m \le \eta |B_i|\le \eta |W_i|$ as $\zeta \ll \eta$.
%Recall that $|N_{A}(v, B_i)| \le 4\sqrt\eta {|B_i|}^{k-\ell-1}$ for any $v\in V_i$.
%So for any $v\in W_i$, we get
%\[
%\deg_{A[W_i]}(v) \le |N_{A}(v, B_i)| + |W_i\setminus B_i| |W_i|^{k-\ell-2} \le (4\sqrt\eta+\eta) {|W_i|}^{k-\ell-1}.
%\]
%Since $\eta\ll 1/k$, by well-known results on perfect matchings, e.g.~\cite[Theorem 3]{DaHa}, there is a perfect matching $M_i$ in the complement of $ A[W_i]$.
%Let $M:= \bigcup_{i\in [t]} M_i$.
We arbitrarily partition $X$ into disjoint sets of size $k-\ell$.
%Note that $|M|\le |B| + |X| + |Y|\le 3t_3\zeta n$.
By the definition of $\cA$, every $(k-\ell)$-set $S\subseteq X$ has at least $\beta^9 n$ $S$-absorbers in $\mathcal A$.
Since each member of $\cA$ is a subpath of $Q$ and $2 t_3 \zeta \le \beta^9$, we can absorb all these $(k-\ell)$-sets greedily and obtain the desired Hamiltonian $\ell$-cycle.

\smallskip
Each of Steps 2, 3 and 4 can be done with probability $1-o(1)$ (while Steps 1 and 5 are deterministic). Hence, by the union bound, \aas~we complete all the steps and obtain a Hamiltonian $\ell$-cycle of~$H$.

\section*{Acknowledgment} \nonumber

We would like to thank Wiebke Bedenknecht, Yoshiharu Kohayakawa and Guilherme Mota for discussions at an early stage of this project.
We are also grateful to two anonymous referees for many helpful comments. %that improved the presentation of the paper.
In particular, we are in debt to a referee who showed us how to obtain an absorbing lemma without using the regularity method. This and other comments helped to simplify our proof and greatly improved the presentation of the paper.

%\newpage

%\appendix

\noindent

\begin{bibdiv}
\begin{biblist}

%\bib{Baranyai}{article}{,
%    AUTHOR = {Baranyai, Z.},
%     TITLE = {On the factorization of the complete uniform hypergraph},
% BOOKTITLE = {Infinite and finite sets ({C}olloq., {K}eszthely, 1973;
%              dedicated to {P}. {E}rd\H os on his 60th birthday), {V}ol. {I}},
%     PAGES = {91--108. Colloq. Math. Soc. J\'an\=os Bolyai, Vol. 10},
% PUBLISHER = {North-Holland, Amsterdam},
%      YEAR = {1975},
%   MRCLASS = {05C99},
%  MRNUMBER = {0416986},
%MRREVIEWER = {D. L. Greenwell},
%}

\bib{BTW}{article}{
   author = {Balogh, J.},
   author = {Treglown, A.},
   author = {Wagner, A.~Z.},
    title = {Tilings in randomly perturbed dense graphs},
  journal={Combinatorics, Probability and Computing}, publisher={Cambridge University Press}, year={2019}, 
  volume={28},
  pages={159-176}
  }

\bib{BMSSS1}{article}{
Author = {Bastos, J.O.},
author={Mota, G. O.},
author={Schacht, M.},
author={Schnitzer, J.},
author= {Schulenburg, F.},
	Date-Added = {2017-02-14 19:33:21 +0000},
	Date-Modified = {2017-02-14 19:33:21 +0000},
	volume={31},
	year={2017},
	pages={2328-2347},
	Journal = {SIAM Journal on Discrete Math.},
	Title = {Loose Hamiltonian cycles forced by large $(k-2)$-degree - approximation version},
	Url = {http://dx.doi.org/10.1137/16M1065732}
}

\bib{BMSSS2}{article}{
	Author = {Bastos, J.O.},
	author={Mota, G. O.},
	author={Schacht, M.},
	author={Schnitzer, J.},
	author= {Schulenburg, F.},
	title = {Loose Hamiltonian cycles forced by large (k-2)-degree - sharp version},
	journal = {Contributions to Discrete Mathematics},
	volume = {13},
	number = {2},
	year = {2018}}

\bib{BHKM}{article}{
	Author = {Bedenknecht, W.},
	author={Han, J.},
	author={Kohayakawa, Y.},
	author={Mota, G. O.},
	Date-Added = {2017-02-14 19:33:21 +0000},
	Date-Modified = {2017-02-14 19:33:21 +0000},
	JOURNAL = {Random Structures \& Algorithms, to appear},
	Title = {Powers of tight Hamilton cycles in randomly perturbed hypergraphs}}
	
\bib{BeDuFr16}{article}{
	author = {Bennett, P.},
	author = {Dudek, A.},
	author =   {Frieze, A.},
	title = {Adding random edges to create the square of a Hamilton cycle},
	journal = {ArXiv e-prints},
	archivePrefix = {arXiv},
	eprint = {1710.02716},
	primaryClass = {math.CO},
	keywords = {Mathematics - Combinatorics},
	year = {2017},
	adsurl = {http://adsabs.harvard.edu/abs/2016arXiv161106570B},
	adsnote = {Provided by the SAO/NASA Astrophysics Data System}
}


\bib{BFM}{article}{
	AUTHOR = {Bohman, Tom},
	author={Frieze, Alan},
	author={Martin, Ryan},
	TITLE = {How many random edges make a dense graph {H}amiltonian?},
	JOURNAL = {Random Structures \& Algorithms},
	FJOURNAL = {Random Structures \& Algorithms},
	VOLUME = {22},
	YEAR = {2003},
	NUMBER = {1},
	PAGES = {33--42},
	ISSN = {1042-9832},
	MRCLASS = {05C80 (05C45 60C05)},
	MRNUMBER = {1943857},
	MRREVIEWER = {Bert Fristedt},
	DOI = {10.1002/rsa.10070},
	URL = {http://dx.doi.org/10.1002/rsa.10070},
}

\bib{BHKMPP}{article}{
	Author = {B\"ottcher, J.},
        author={Han, J.},
	author={Kohayakawa, Y.},
        author={Montgomery, R.},
	author={Parczyk, O.},
	author={Person, Y.},
	Date-Added = {2017-02-14 19:33:21 +0000},
	Date-Modified = {2017-02-14 19:33:21 +0000},
	Journal = {Random Structures \& Algorithms, to appear},
	Title = {Universality of bounded degree spanning trees in randomly perturbed graphs}}


\bib{BMPP}{article}{
	Author = {B\"ottcher, J.},
	author={Montgomery, R.},
	author={Parczyk, O.},
	author={Person, Y.},
	Date-Added = {2017-02-14 19:33:21 +0000},
	Date-Modified = {2017-02-14 19:33:21 +0000},
	Journal = {preprint},
	Title = {Embedding spanning bounded degree subgraphs in randomly perturbed graphs}}



\bib{BHS}{article}{
	Author = {Bu{\ss}, E.},
	author={H{\`a}n, H.},
	author={Schacht, M.},
	Date-Added = {2017-02-14 19:33:33 +0000},
	Date-Modified = {2017-02-14 19:33:33 +0000},
	Doi = {10.1016/j.jctb.2013.07.004},
	Fjournal = {Journal of Combinatorial Theory. Series B},
	Issn = {0095-8956},
	Journal = {J. Combin. Theory Ser. B},
	Mrclass = {05C65 (05C45)},
	Mrnumber = {3127586},
	Mrreviewer = {Martin Sonntag},
	Number = {6},
	Pages = {658--678},
	Title = {Minimum vertex degree conditions for loose {H}amilton cycles in 3-uniform hypergraphs},
	Url = {http://dx.doi.org/10.1016/j.jctb.2013.07.004},
	Volume = {103},
	Year = {2013},
	Bdsk-Url-1 = {http://dx.doi.org/10.1016/j.jctb.2013.07.004}}

%\bib{CFKO}{article}{,
%	Author = {Cooley, O.},
%	author = {Fountoulakis, N.},
%	author = {K{\"u}hn, D.},
%	author = {Osthus, D.},
%	Date-Added = {2017-02-14 19:33:33 +0000},
%	Date-Modified = {2017-02-14 19:33:33 +0000},
%	Doi = {10.1007/s00493-009-2356-y},
%	Fjournal = {Combinatorica. An International Journal on Combinatorics and the Theory of Computing},
%	Issn = {0209-9683},
%	Journal = {Combinatorica},
%	Mrclass = {05C55 (05C65 05D10)},
%	Mrnumber = {2520273 (2010g:05241)},
%	Mrreviewer = {Andr{\'a}s Gy{\'a}rf{\'a}s},
%	Number = {3},
%	Pages = {263--297},
%	Title = {Embeddings and {R}amsey numbers of sparse {$k$}-uniform hypergraphs},
%	Url = {http://dx.doi.org/10.1007/s00493-009-2356-y},
%	Volume = {29},
%	Year = {2009},
%	Bdsk-Url-1 = {http://dx.doi.org/10.1007/s00493-009-2356-y}}


\bib{CzMo}{article}{
	Author = {Czygrinow, A.},
	author={Molla, T.},
	Date-Added = {2017-02-14 19:33:33 +0000},
	Date-Modified = {2017-02-14 19:33:33 +0000},
	Doi = {10.1137/120890417},
	Fjournal = {SIAM Journal on Discrete Mathematics},
	Issn = {0895-4801},
	Journal = {SIAM J. Discrete Math.},
	Mrclass = {05D40 (05C65)},
	Mrnumber = {3150175},
	Mrreviewer = {Deryk Osthus},
	Number = {1},
	Pages = {67--76},
	Title = {Tight codegree condition for the existence of loose {H}amilton cycles in 3-graphs},
	Url = {http://dx.doi.org/10.1137/120890417},
	Volume = {28},
	Year = {2014},
	Bdsk-Url-1 = {http://dx.doi.org/10.1137/120890417}}


%\bib{DaHa}{article}{,
%	Author = {Daykin, D. E.},
%	author = {H{\"a}ggkvist, R.},
%	Coden = {ALNBAB},
%	Date-Added = {2017-02-14 19:33:33 +0000},
%	Date-Modified = {2017-02-14 19:33:33 +0000},
%	Doi = {10.1017/S0004972700006924},
%	Fjournal = {Bulletin of the Australian Mathematical Society},
%	Issn = {0004-9727},
%	Journal = {Bull. Austral. Math. Soc.},
%	Mrclass = {05C65 (05C35)},
%	Mrnumber = {615135},
%	Mrreviewer = {Andr{\'a}s Frank},
%	Number = {1},
%	Pages = {103--109},
%	Title = {Degrees giving independent edges in a hypergraph},
%	Url = {http://dx.doi.org/10.1017/S0004972700006924},
%	Volume = {23},
%	Year = {1981},
%	Bdsk-Url-1 = {http://dx.doi.org/10.1017/S0004972700006924}}


\bib{Di52}{article}{
	author = {Dirac, G. A.},
	title = {Some Theorems on Abstract Graphs},
	journal = {Proceedings of the London Mathematical Society},
	volume = {s3-2},
	number = {1},
	publisher = {Oxford University Press},
	issn = {1460-244X},
	url = {http://dx.doi.org/10.1112/plms/s3-2.1.69},
	doi = {10.1112/plms/s3-2.1.69},
	pages = {69--81},
	year = {1952},
}

\bib{DuFr2}{article}{
	AUTHOR = {Dudek, Andrzej},
	author={Frieze, Alan},
	TITLE = {Loose {H}amilton cycles in random uniform hypergraphs},
	JOURNAL = {Electron. J. Combin.},
	FJOURNAL = {Electronic Journal of Combinatorics},
	VOLUME = {18},
	YEAR = {2011},
	NUMBER = {1},
	PAGES = {Paper 48, 14},
	ISSN = {1077-8926},
	MRCLASS = {05C80 (05C45 05C65)},
	MRNUMBER = {2776824},
}


\bib{DuFr1}{article}{
	AUTHOR = {Dudek, Andrzej},
	author={Frieze, Alan},
	TITLE = {Tight {H}amilton cycles in random uniform hypergraphs},
	JOURNAL = {Random Structures \& Algorithms},
	FJOURNAL = {Random Structures \& Algorithms},
	VOLUME = {42},
	YEAR = {2013},
	NUMBER = {3},
	PAGES = {374--385},
	ISSN = {1042-9832},
	MRCLASS = {05C80 (05C45 05C65)},
	MRNUMBER = {3039684},
	MRREVIEWER = {Andrew Clark Treglown},
	DOI = {10.1002/rsa.20404},
	URL = {http://dx.doi.org/10.1002/rsa.20404},
}


%\bib{erdos}{article}{
%	Author = {Erd\H{o}s, P.},
%	Date-Added = {2017-02-14 19:33:33 +0000},
%	Date-Modified = {2017-02-14 19:33:33 +0000},
%	Doi = {10.1007/BF02759942},
%	Issn = {0021-2172},
%	Journal = {Israel Journal of Mathematics},
%	Language = {English},
%	Number = {3},
%	Pages = {183-190},
%	Publisher = {Springer-Verlag},
%	Title = {On extremal problems of graphs and generalized graphs},
%	Url = {http://dx.doi.org/10.1007/BF02759942},
%	Volume = {2},
%	Year = {1964},
%	Bdsk-Url-1 = {http://dx.doi.org/10.1007/BF02759942}}


\bib{GPW}{article}{
	Author = {Glebov, R.},
	author={Person, Y.},
	author={Weps, W.},
	Date-Added = {2017-02-14 19:33:33 +0000},
	Date-Modified = {2017-02-14 19:33:33 +0000},
	Doi = {10.1016/j.ejc.2011.10.003},
	Fjournal = {European Journal of Combinatorics},
	Issn = {0195-6698},
	Journal = {European J. Combin.},
	Mrclass = {05C35 (05C45 05C65)},
	Mrnumber = {2864440},
	Mrreviewer = {Martin Sonntag},
	Number = {4},
	Pages = {544--555},
	Title = {On extremal hypergraphs for {H}amiltonian cycles},
	Url = {http://dx.doi.org/10.1016/j.ejc.2011.10.003},
	Volume = {33},
	Year = {2012},
	Bdsk-Url-1 = {http://dx.doi.org/10.1016/j.ejc.2011.10.003}}

%\bib{Han14_Poly}{article}{,
%	Author = {Han, J.},
%	Date-Added = {2017-02-14 19:33:33 +0000},
%	Date-Modified = {2017-02-14 19:33:33 +0000},
%	Journal = {Trans. Amer. Math. Soc.},
%        Volume = {369},
%        Number = {7},
%	Year = {2017},
%	Pages = {5197--5218},
%	Title = {Decision problem for perfect matchings in dense $k$-uniform hypergraphs}}
	
	
\bib{HZ2}{article}{
	Author = {Han, J.},
	author={Zhao, Y.},
	Date-Added = {2017-02-14 19:33:21 +0000},
	Date-Modified = {2017-02-14 19:33:21 +0000},
	Doi = {http://dx.doi.org/10.1016/j.jcta.2015.01.004},
	Issn = {0097-3165},
	Journal = {J. Combin. Theory Ser. A},
	Keywords = {Regularity lemma},
	Number = {0},
	Pages = {194 - 223},
	Title = {Minimum codegree threshold for Hamilton $\ell$-cycles in k-uniform hypergraphs},
	Url = {http://www.sciencedirect.com/science/article/pii/S0097316515000059},
	Volume = {132},
	Year = {2015},
	Bdsk-Url-1 = {http://www.sciencedirect.com/science/article/pii/S0097316515000059},
	Bdsk-Url-2 = {http://dx.doi.org/10.1016/j.jcta.2015.01.004}}	

\bib{HZ1}{article}{
	Author = {Han, J.},
	author={Zhao, Y.},
	Date-Added = {2017-02-14 19:33:21 +0000},
	Date-Modified = {2017-02-14 19:33:21 +0000},
	Journal = {J. Combin. Theory Ser. B},
	Pages = {70 - 96},
	Title = {Minimum degree thresholds for loose {Hamilton} cycle in 3-graphs},
	Volume = {114},
	Year = {2015}}	
	
%\bib{HPS}{article}{
%	Author = {H\`an, H.},
%	author = {Person, Y.},
%	author = {Schacht, M.},
%	Date-Added = {2017-02-14 19:33:33 +0000},
%	Date-Modified = {2017-02-14 19:33:33 +0000},
%	Issue = {2},
%	Journal = {SIAM J. Discrete Math},
%	Pages = {732--748},
%	Title = {On perfect matchings in uniform hypergraphs with large minimum vertex degree},
%	Volume = {23},
%	Year = {2009}}

\bib{HS}{article}{
	Author = {H\`an, H.},
	author= {Schacht, M.},
	Date-Added = {2017-02-14 19:33:21 +0000},
	Date-Modified = {2017-02-14 19:33:21 +0000},
	Issue = {3},
	Journal = {J. Combin. Theory Ser. B},
	Pages = {332--346},
	Title = {Dirac-type results for loose {Hamilton} cycles in uniform hypergraphs},
	Volume = {100},
	Year = {2010}}


\bib{JLR}{book}{
	AUTHOR = {Janson, Svante},
	author={\L uczak, Tomasz},
	author={Rucinski, Andrzej},
	TITLE = {Random graphs},
	SERIES = {Wiley-Interscience Series in Discrete Mathematics and
		Optimization},
	PUBLISHER = {Wiley-Interscience, New York},
	YEAR = {2000},
	PAGES = {xii+333},
	ISBN = {0-471-17541-2},
	MRCLASS = {05C80 (60C05 82B41)},
	MRNUMBER = {1782847},
	MRREVIEWER = {Mark R. Jerrum},
	DOI = {10.1002/9781118032718},
	URL = {http://dx.doi.org/10.1002/9781118032718},
}


\bib{Karp}{article}{
	author = {Karp, Richard M.},
	TITLE = {Reducibility among combinatorial problems},
	BOOKTITLE = {Complexity of computer computations ({P}roc. {S}ympos., {IBM}
		{T}homas {J}. {W}atson {R}es. {C}enter, {Y}orktown {H}eights,
		{N}.{Y}., 1972)},
	PAGES = {85--103},
	PUBLISHER = {Plenum, New York},
	YEAR = {1972},
	MRCLASS = {68A20},
	MRNUMBER = {0378476},
	MRREVIEWER = {John T. Gill},
}



\bib{KKMO}{article}{
	Author = {Keevash, P.},
	author={K\"uhn, D.},
	author= {Mycroft, R.},
	author= {Osthus, D.},
	Date-Added = {2017-02-14 19:33:21 +0000},
	Date-Modified = {2017-02-14 19:33:21 +0000},
	Journal = {Discrete Math.},
	Number = {7},
	Pages = {544--559},
	Title = {Loose {Hamilton} cycles in hypergraphs},
	Volume = {311},
	Year = {2011}}


\bib{Korshunov} {article}{
	AUTHOR = {Kor\v sunov, A. D.},
	TITLE = {Solution of a problem of {P}. {E}rd\H os and {A}. {R}\'enyi on
		{H}amiltonian cycles in nonoriented graphs},
	JOURNAL = {Diskret. Analiz},
	NUMBER = {31 Metody Diskret. Anal. v Teorii Upravljaju\v s\v cih Sistem},
	YEAR = {1977},
	PAGES = {17--56, 90},
	MRCLASS = {05C35},
	MRNUMBER = {0543833},
}


\bib{KKS}{article}{
	AUTHOR = {Krivelevich, Michael},
	author={Kwan, Matthew},
	author={Sudakov, Benny},
	TITLE = {Cycles and matchings in randomly perturbed digraphs and
		hypergraphs},
	JOURNAL = {Combin. Probab. Comput.},
	FJOURNAL = {Combinatorics, Probability and Computing},
	VOLUME = {25},
	YEAR = {2016},
	NUMBER = {6},
	PAGES = {909--927},
	ISSN = {0963-5483},
	MRCLASS = {05C80 (05C35 05C65)},
	MRNUMBER = {3568952},
	DOI = {10.1017/S0963548316000079},
	URL = {http://dx.doi.org/10.1017/S0963548316000079},
}

\bib{KKS2}{article}{
AUTHOR = {Krivelevich, Michael},
	author={Kwan, Matthew},
	author={Sudakov, Benny},
  title={Bounded-degree spanning trees in randomly perturbed graphs},
  journal={SIAM Journal on Discrete Mathematics},
  volume={31},
  number={1},
  pages={155--171},
  year={2017},
  publisher={SIAM}
}


\bib{KMO}{article}{
	author={K\"uhn, D.},
	author= {Mycroft, R.},
	author= {Osthus, D.},
	Date-Added = {2017-02-14 19:33:21 +0000},
	Date-Modified = {2017-02-14 19:33:21 +0000},
	Journal = {J. Combin. Theory Ser. A},
	Number = {7},
	Pages = {910--927},
	Title = {Hamilton $\ell$-cycles in uniform hypergraphs},
	Volume = {117},
	Year = {2010}}


\bib{KO}{article}{
	Author = {K\"uhn, D.},
	author= {Osthus, D.},
	Date-Added = {2017-02-14 19:33:21 +0000},
	Date-Modified = {2017-02-14 19:33:21 +0000},
	Journal = {J. Combin. Theory Ser. B},
	Number = {6},
	Pages = {767--821},
	Title = {Loose {Hamilton} cycles in 3-uniform hypergraphs of high minimum degree},
	Volume = {96},
	Year = {2006}}	


\bib{McMy}{article}{
	Author = {McDowell, A.},
	author={Mycroft, R.},
	Date-Added = {2017-02-14 19:33:21 +0000},
	Date-Modified = {2017-02-14 19:33:21 +0000},
	JOURNAL = {Electron. J. Combin.},
	FJOURNAL = {Electronic Journal of Combinatorics},
	VOLUME = {25},
	YEAR = {2018},
	PAGES = {P4.36},
	ISSN = {1077-8926},
	Title = {Hamilton {$\ell$}-cycles in randomly perturbed hypergraphs}}


\bib{Posa}{article}{
  AUTHOR = {P\'osa, L.},
TITLE = {Hamiltonian circuits in random graphs},
JOURNAL = {Discrete Math.},
FJOURNAL = {Discrete Mathematics},
VOLUME = {14},
YEAR = {1976},
NUMBER = {4},
PAGES = {359--364},
ISSN = {0012-365X},
MRCLASS = {05C35},
MRNUMBER = {0389666},
MRREVIEWER = {F. Harary},
DOI = {10.1016/0012-365X(76)90068-6},
URL = {http://dx.doi.org/10.1016/0012-365X(76)90068-6},
}

\bib{RRRSS}{article}{
       author = {{Reiher}, Christian},
       author={{R{\"o}dl}, Vojt{\v{e}}ch},
       author={Ruci{\'n}ski, Andrzej},
       author={{Schacht}, Mathias},
       author={ {Szemer{\'e}di}, Endre},
        title = {Minimum vertex degree condition for tight Hamiltonian cycles in 3-uniform hypergraphs},
      journal = {Proceedings of the London Mathematical Society, to appear}
      }
      
\bib{RR}{book} {
    AUTHOR = {R\"{o}dl, V.},
    author={Ruci\'{n}ski, Andrzej},
     TITLE = {Dirac-type questions for hypergraphs---a survey (or more
              problems for {E}ndre to solve)},
 BOOKTITLE = {An irregular mind},
    SERIES = {Bolyai Soc. Math. Stud.},
    VOLUME = {21},
     PAGES = {561--590},
 PUBLISHER = {J\'{a}nos Bolyai Math. Soc., Budapest},
      YEAR = {2010},
   MRCLASS = {05-02 (05C65 05C70)},
  MRNUMBER = {2815614},
       DOI = {10.1007/978-3-642-14444-8_16},
       URL = {https://doi.org/10.1007/978-3-642-14444-8_16},
}

\bib{RoRu14}{article}{
	Author = {R{\"o}dl, V.},
	author={Ruci{\'n}ski, A.},
	Date-Added = {2017-02-14 19:33:33 +0000},
	Date-Modified = {2017-02-14 19:33:33 +0000},
	Doi = {10.7151/dmgt.1743},
	Fjournal = {Discussiones Mathematicae. Graph Theory},
	Issn = {1234-3099},
	Journal = {Discuss. Math. Graph Theory},
	Mrclass = {05D05 (05C65)},
	Mrnumber = {3194042},
	Mrreviewer = {Peter James Dukes},
	Number = {2},
	Pages = {361--381},
	Title = {Families of triples with high minimum degree are {H}amiltonian},
	Url = {http://dx.doi.org/10.7151/dmgt.1743},
	Volume = {34},
	Year = {2014},
	Bdsk-Url-1 = {http://dx.doi.org/10.7151/dmgt.1743}}

%\bib{RRS06}{article}{
%	Author = {R\"odl, V.},
%	author={Ruci\'nski, A.},
%	author={Szemer\'edi, E.},
%	Date-Added = {2017-02-14 19:33:33 +0000},
%	Date-Modified = {2017-02-14 19:33:33 +0000},
%	Fjournal = {Combinatorics, Probability and Computing},
%	Journal = {Combin. Probab. Comput.},
%	Number = {1-2},
%	Pages = {229--251},
%	Title = {A {D}irac-type theorem for 3-uniform hypergraphs},
%	Volume = {15},
%	Year = {2006}}


\bib{RoRuSz06}{article}{
	Author = {R\"odl, V.},
author={Ruci\'nski, A.},
author={Szemer\'edi, E.},
	title={A Dirac-type theorem for 3-uniform hypergraphs},
	journal={Combin. Probab. Comput.},
	volume={15},
	date={2006},
	number={1-2},
	pages={229--251},
	issn={0963-5483},
	review={\MR{2195584}},
	doi={10.1017/S0963548305007042},
}


\bib{RRS08}{article}{
	Author = {R\"odl, V.},
author={Ruci\'nski, A.},
author={Szemer\'edi, E.},
	Date-Added = {2017-02-14 19:33:33 +0000},
	Date-Modified = {2017-02-14 19:33:33 +0000},
	Journal = {Combinatorica},
	Number = {2},
	Pages = {229--260},
	Title = {An approximate {D}irac-type theorem for k-uniform hypergraphs},
	Volume = {28},
	Year = {2008}}

\bib{RRS11}{article}{
		Author = {R\"odl, V.},
	author={Ruci\'nski, A.},
	author={Szemer\'edi, E.},
	Date-Added = {2017-02-14 19:33:33 +0000},
	Date-Modified = {2017-02-14 19:33:33 +0000},
	Journal = {Advances in Mathematics},
	Number = {3},
	Pages = {1225--1299},
	Title = {Dirac-type conditions for {Hamiltonian} paths and cycles in 3-uniform hypergraphs},
	Volume = {227},
	Year = {2011}}
	
%\bib{RoSc}{article}{,
%	Author = {R{\"o}dl, V.},
%	author = {Schacht, M.},
%	Date-Added = {2017-02-14 19:33:33 +0000},
%	Date-Modified = {2017-02-14 19:33:33 +0000},
%	Fjournal = {Combinatorics, Probability and Computing},
%	Issn = {0963-5483},
%	Journal = {Combin. Probab. Comput.},
%	Mrclass = {05C65 (05C35 05C70 05D05 05D10)},
%	Mrnumber = {2351688 (2008h:05083)},
%	Mrreviewer = {G{\'a}bor N. S{\'a}rk{\"o}zy},
%	Number = {6},
%	Pages = {833--885},
%	Title = {Regular partitions of hypergraphs: regularity lemmas},
%	Volume = {16},
%	Year = {2007}}


\bib{SpTe}{article}{
AUTHOR = {Spielman, Daniel A.},
author= {Teng, Shang-Hua},
TITLE = {Smoothed analysis: motivation and discrete models},
BOOKTITLE = {Algorithms and data structures},
SERIES = {Lecture Notes in Comput. Sci.},
VOLUME = {2748},
PAGES = {256--270},
PUBLISHER = {Springer, Berlin},
YEAR = {2003},
MRCLASS = {68W40 (05C85)},
MRNUMBER = {2078601},
}

\bib{zsurvey}{book}{
    AUTHOR = {Zhao, Yi},
     TITLE = {Recent advances on {D}irac-type problems for hypergraphs},
 BOOKTITLE = {Recent trends in combinatorics},
    SERIES = {IMA Vol. Math. Appl.},
    VOLUME = {159},
     PAGES = {145--165},
 PUBLISHER = {Springer, [Cham]},
      YEAR = {2016},
   MRCLASS = {05-02 (05C35 05C45 05C65 05C70)},
  MRNUMBER = {3526407},
       DOI = {10.1007/978-3-319-24298-9_6},
       URL = {https://doi.org/10.1007/978-3-319-24298-9_6},
}

	
\end{biblist}
\end{bibdiv}
\end{document}